\documentclass[12pt]{article}

\usepackage{amsmath,amssymb}

\def\essinf{\mathop{\mathrm{essinf}}\limits}
\def\esssup{\mathop{\mathrm{esssup}}\limits}

\def\Dzw1#1{\frac{\partial^2 #1}{\partial z \partial w_1}}

\def\Dzb1#1{\frac{\partial^2 #1}{\partial z \partial b_1}}

\numberwithin{equation}{section}

\usepackage{amsthm}
\newtheorem{definition}{Definition}[section]

\newtheorem{theorem}[definition]{Theorem}

\newtheorem{remark}[definition]{ \it Remark}

\newtheorem{proposition}[definition]{Proposition}
\newtheorem{lemma}[definition]{Lemma}

\newtheorem{assumption}[definition]{Assumption}

\def\1B{\text{1\!\!I}}

\def\1B{\text{1\!\!I}}

\def\O{\Omega}

\def\o{\omega}

\def\a{\alpha}

\def\t{\theta}

\def\cal#1{\mathcal{#1}}
\def\mb#1{\mathbf{#1}}

\def\dd{\displaystyle}

\begin{document}

\title{BSDEs with nonlinear weak terminal condition.}
\author{Roxana DUMITRESCU\thanks{Institut fur Mathematik, Humboldt-Universitat zu Berlin, Unter den Linden 6, 10099 Berlin, Germany, email: {\tt roxana@ceremade.dauphine.fr}. The author takes the opportunity to express her gratitude to  Bruno Bouchard and Romuald Elie for fruitful discussions.}} 
%

\maketitle

\begin{abstract}
In a recent paper, Bouchard, Elie and Reveillac  \cite{BER} have studied a new class of Backward Stochastic Differential Equations with weak terminal condition, for which the $T$-terminal value $Y_T$ of the solution $(Y,Z)$ is not fixed as a random variable, but only satisfies a constraint of the form
$
E[\Psi(Y_T)] \geq m.
$
The aim of this paper is to introduce a more general class of BSDEs with { \em nonlinear weak terminal condition}. More precisely, the constraint takes the form
$
\mathcal{E}^f_{0,T}[\Psi(Y_T)] \geq m,
$
 where $\mathcal{E}^f$ represents the $f$-conditional expectation associated to   a { \em nonlinear driver} $f$. 
 We carry out a similar analysis as in \cite{BER} of the  value function corresponding to the minimal solution  $Y$ of the BSDE with nonlinear weak terminal condition: we study the  regularity, establish the main properties, in particular continuity and convexity with respect to the parameter $m$, and  finally provide  a dual representation and the existence of an optimal control in the case of concave constraints. From a financial point of view, our study is closely related to  the approximative hedging  of an European option under dynamic risk measures constraints.  The nonlinearity $f$ raises subtle difficulties, highlighted throughout the paper, which cannot be handled by the arguments used in the case of classical expectations constraints studied in \cite{BER}.

\end{abstract}
%
%
%
\noindent{\bf Key words~:}   Backward stochastic differential equations, $g$-expectation, dynamic risk measures, optimal control, stochastic targets.

  \section{Introduction}
Linear backward stochastic differential equations (BSDEs) were introduced by Bismut as the adjoint equations associated with Pontryagin  maximum principles in stochastic control theory. The general case of non-linear BSDEs was then studied by Pardoux and Peng \cite{P}. They provided Feynman-Kac representations  of solutions  of non-linear parabolic partial differential equations.\\
  \quad  The solution of a BSDE consists in a pair of predictable processes $(Y,Z)$ satisfying
  \begin{align}
  -dY_t=g(t,Y_t,Z_t)dt-Z_tdW_t; \,\,\,\,  Y_T=\xi.
  \end{align}
These equations appear as an useful mathematical tool in various problems in finance, for example in the theory of derivatives pricing. In a complete market - when it is possible to construct a portfolio which attains as final wealth the payoff- the value of the replicating portfolio is given by $Y$ and the hedging strategy by $Z$. Since in incomplete markets is not always possible to construct a portfolio which attains exactly as final wealth the amount $\xi$, it was suggested to replace the terminal condition into a weaker one of the form $Y_T \geq \xi.$ In this case, the minimal initial value $Y_0$ defines the smallest initial investment which allows one to superhedge the contingent claim $\xi$.\\
\quad Recently, Bouchard, Elie and Reveillac  \cite{BER}  introduced  a new class of BSDEs, the so called BSDEs with weak terminal condition, in which the $T$-terminal value $Y_T$  only satisfies a weak constraint. More precisely, a couple of predictable processes $(Y,Z)$ is said to be a solution of  such a  BSDE  if it satisfies:
\begin{equation}\label{cond11}
-dY_t=g(t,Y_t,Z_t)dt-Z_tdW_t;
\end{equation}
\begin{equation}\label{cond12}
E[\Psi(Y_T)] \geq m,
\end{equation}
where $m$ is a given threshold and $\Psi$ a non-decreasing map. The main question in \cite{BER} is the following:
\begin{align}\label{initial}
  \text{Find the minimal } Y_0 \text{ such that } \eqref{cond11} \text{ and } \eqref{cond12} \text{ hold for some } Z.
\end{align}
 From a financial point of view, this study is related to the hedging in quantile or more generally  to the hedging with \textit{expected loss constraints}. This problem was addressed in the literature for the first time by F\"{o}llmer and Leukert \cite{FL1} and then  further studied in a Markovian framework in $\cite{BET}$ and $\cite{M}$, using stochastic target techniques .\\
In \cite{BER}, the key point of the analysis is the reformulation of the problem written in terms of BSDE with weak terminal condition  into  an optimization problem on a family of BSDEs with strong terminal condition, by using { \em the martingale representation theorem}. The main observation is that if $Y_0$ and $Z$ are such that \eqref{cond12} holds, then the martingale representation Theorem implies that it exists an element $\alpha \in \textbf{A}_0$, the set of predictable square integrable processes, such that:
  \begin{align}
 \Psi(Y_T) \geq  M_T^{m,\alpha}=m+\int_0^T \alpha_sdW_s,
  \end{align}

 It is then shown that the initial problem \eqref{initial} is equivalent to:
  \begin{align}\label{equivproblem}
   \inf \{Y_0^\alpha, \,\, \alpha \in \textbf{A}_0 \},
   \end{align}
   where   $Y_t^\alpha$ corresponds to the solution at time $t$ of the BSDE with (strong) terminal condition $\Phi(M_T^\alpha)$, $\Phi$ being the left-continuous inverse of $\Psi$.\\

   \quad The aim of this paper is to introduce a new class of { \em BSDEs with weak nonlinear  terminal condition }. We extend the results of \cite{BER} to a more general class of constraints that take the form:
   \begin{align}\label{cond13}
   \mathcal{E}_{0,T}^f[\Psi(Y_T)] \geq m,
   \end{align}
   where $f$ is a nonlinear driver and  $\mathcal{E}^f_{\cdot,T}[\xi]$ the solution of the BSDE with generator $f$ and terminal condition $\xi$.

  We can easily remark that the constraint \eqref{cond12} is a particular case of \eqref{cond13} for $f=0$.\\
   The problem under study in this paper is the following:
   \begin{align}\label{problem}
  \inf \{Y_0 \text{ such that } \exists Z:\, (\ref{cond11}) \text{  and } (\ref{cond13}) \text{ hold}\}.
   \end{align}

   Following the key idea of \cite{BER}, we rewrite our problem \eqref{problem} into an equivalent one expressed in terms of BSDEs with strong terminal condition. The main difference with respect to \cite{BER} is given by the fact that in our case we have to introduce a new controlled diffusion process, which is an {\em $f-$martingale}, contrary to \cite{BER} where it is a classical martingale. Indeed, for a given $Y_0$ and $Z$ such that \eqref{cond11} and \eqref{cond13} are satisfied, appealing to the {\em BSDE representation} of $\Psi(Y_T)$, we can find $\alpha \in \textbf{A}_0$ such that:
   \begin{align}
   \Psi(Y_T) \geq  \mathcal{M}_T^{m,\alpha}=m-\int_0^Tf(s, \mathcal{M}^{m,\alpha}_s,\alpha_s)ds+\int_0^T \alpha_sdW_s.
   \end{align}
  Thanks to this observation,  we show that Problem \eqref{problem} is equivalent to \eqref{equivproblem}, where, in our more general framework,  $Y_t^\alpha$ corresponds to the solution at time $t$ of the BSDE with (strong) terminal condition $\Phi(\mathcal{M}_T^\alpha)$.
We study the dynamical counterpart of \eqref{equivproblem}:
\begin{align}
\mathcal{Y}^\alpha(\tau):= \essinf \{Y_\tau^{\alpha'}, \alpha' \in \textbf{A}_0 \text{ s.t.} \alpha'=\alpha \text{ on } [\![ 0,\tau ]\!]\}.
\end{align}
We carry out a similar analysis as in \cite{BER} of the family $\{\mathcal{Y}^\alpha, \,\, \alpha \in \textbf{A}_0\}.$ We start by studying the regularity of the family $\mathcal{Y}^\alpha$ and show that it can be aggregated into a RCLL process, proof which becomes considerably more technical in our context with respect to \cite{BER}, because we have to deal with the nonlinearity $f$. We then provide a BSDE representation of $\mathcal{Y}^\alpha$ and show that, under a concavity assumption on the driver $f$, there exists an optimal control. We also study the main properties of the value function, as continuity and convexity with respect to the threshold $m$, and propose proofs specific to the nonlinear case.We finally get, in the case of concave constraints, a dual representation of the value function,  related to a stochastic control problem in Meyer's form.  We point out that the techniques used in \cite{BER} cannot be adapted to our framework. 

  Besides the mathematical interest of our study, this work is also motivated by some financial applications, as it provides  the $\textit{approximative hedging  under dynamic risk}$\\  $\textit{measures contraints }$ of an European option, when the shortfall risk is quantified in terms of dynamic risk measures induced by BSDEs (see e.g. \cite{BEK}, \cite{15}).
  
%

  The paper is organized as follows. In Section 2 we introduce notation, assumptions and the BSDEs with nonlinear weak terminal condition. In Section 3, we study the regularity and the BSDE representation of the value function $\mathcal{Y}^\alpha$. In Section 4, we provide the existence of an optimal control under some additional assumptions on the coefficients. In Section 5,  we establish the main properties of the value function and we finally  provide a dual representation in Section 6. 
\section{Problem formulation}
\subsection{Notation}
Let $(\Omega,\mathcal{F}, \mathbb{P})$ be a probability space supporting a  $d$-dimensional Brownian motion $W$ and  $\mathbb{F}:=(\mathcal{F}_t)_{t \leq T}$ the completed associated filtration. Fix $T>0.$

In the sequel, we adopt the following notation:

\begin{itemize}

\item[$-$] $\mathcal{P}$ denotes the predictable $\sigma$-algebra on $[0,T] \times \Omega;$

\item[$-$] For any $\sigma$-algebra $\mathcal{G} \subset \mathcal{F}_T$,  $\mathbf{L}_2(\mathcal{G})$ is the set of random variables $\xi$ which are $\mathcal{G}$-measurable and {square}-integrable;

\item[$-$]

$\mathbf{H}_2$ denotes the set of $\mathbf{R}^d$-valued predictable processes $\phi$ such that\\ $\|\phi\|^{2}_{\mathbf{H}_2}: = E[(\int_{0}^{T}\phi_{t}^{2}dt)] <\infty$;

\item[$-$]
$\mathbf{S}_2$ is the set of real-valued RCLL adapted processes $\phi$ such that\\ $\|\phi\|_{\textbf{S}^{2}}^{2} : = E [\sup_{0\leq t\leq T} |\phi_{t}|^{2}]<\infty$;

\item[$-$]
$\mathbf{I}_2$ is the set of non-decreasing adapted processes $\phi$ such that $\|\phi\|_{\textbf{S}^{2}}^{2} <\infty$;

\item[$-$]
For any $\sigma$-algebra $\mathcal{G} \subset \mathcal{F}_T$, $\textbf{L}_0(\mathcal{G})$ denotes the set of random variables measurable with respect to $\mathcal{G}$;


\item[$-$] $\mathcal{T}$ denotes the set of stopping times $\tau$ such that $\tau \in [0,T]$ a.s.

\end{itemize}

\subsection{BSDEs with nonlinear weak terminal condition.}
\subsubsection{Definition and Assumptions.}

\quad In this section, we introduce the main object of this paper, the \textit{BSDEs with nonlinear weak  terminal condition.}\\
\quad It is well known that, in the  case of nonlinear backward stochastic differential equations\\ ( in short BSDEs) introduced by Pardoux-Peng, the data of the BSDE is represented by a {\em driver} $g$ and a {\em terminal condition } $\xi$.

\quad In the recent paper \cite{BER}, the authors define a new class of BSDEs called $\textit{BSDEs with weak}$\\ $\textit{terminal condition}$. The particularity consists in the fact that the terminal condition is not fixed as a $\mathcal{F}_T$-measurable random variable, but only satisfies a weak constraint expressed in terms of classical expectations. The data of this class of BSDEs is given by four elements: a  {\em driver} $g$ and a triplet $(\Psi,\mu,\tau)$ describing  the constraint on the terminal condition.

\quad The aim of this work is to introduce a more general class of BSDEs, named $\textit{BSDEs with }$ $\textit{nonlinear weak terminal condition}$, whose terminal value verifies a weak constraint defined via a BSDE with a nonlinear driver $f$, satisfying the following hypothesis:

\begin{assumption}\label{f}
Let
$f:(\omega,t,y,z) \in \Omega \times [0,T] \times \mathbf{R} \times \mathbf{R}^d \mapsto f_t(\omega,y,z) \in \mathbf{R}$  be a driver such that $\left(f_t(\cdot,y,z)\right)_{t \leq T}$ is $\mathcal{P}$-measurable for every $(y,z) \in \mathbf{R} \times \mathbf{R}^d$ and
$$\left|f_t(\omega,y,z)-f_t(\omega,y',z')\right| \leq C_f \left(|y-y'|+\|z-z'\|_{\mathbf{R}^d}\right),$$
$\forall (y,z), (y',z') \in \mathbf{R} \times \mathbf{R}^d$, for  $dt \otimes dP$-a.e. $(t,\omega) \in [0,T] \times \Omega$, for some constant number $C_f>0.$\\
We also assume that $f$ satisfies the following condition
$$\textbf{E}\left[\int_0^T |f_t(0,0)|^2 dt\right]< \infty.$$

\end{assumption}
Note that the data of this new BSDE are $(f,\Psi,\mu,\tau,g)$ and  the particular case when $f=0$ corresponds to the class of BSDEs studied in \cite{BER}. In the sequel, we shall denote the BSDE with nonlinear weak terminal condition by $BSDE(f,\Psi,\mu,\tau,g)$.\\
\quad Before defining this new mathematical object, we   introduce the nonlinear conditional expectation $\mathcal{E}^f$ associated with $f$, defined for each stopping time $\tau \in \mathcal{T}$ and for each $\eta \in \textbf{L}_2(\mathcal{F}_\tau)$ as:
\begin{align}
\mathcal{E}^{f}_{t,\tau}[\eta]:=Y_r, \quad \quad 0 \leq t \leq \tau,
\end{align}
where $(Y_t)_{t \leq \tau}$ is the unique solution in $\textbf{S}_2$ of the BSDE associated with driver $f$, terminal time $\tau$ and terminal condition $\eta$, that is satisfying:
\begin{align}
\begin{cases}
-dY_t= f(t,Y_t,Z_t)dt-Z_t dW_t;\\
Y_\tau=\eta,
\end{cases}
\end{align}
with $Z$ the associated process belonging to $\textbf{H}_2$.
 Moreover, set $\mathcal{E}_{\sigma,\tau}^f[\eta]:= -\infty,$ for any $\eta \in \mathbf{L}_0(\mathcal{F}_\tau)$ such that $\textbf{E}[\eta^-]=+\infty$, where $\sigma \in \mathcal{T}$ with $\sigma \leq \tau$ a.s.

We are now in position to define the so-called \textit{BSDEs with nonlinear weak terminal condition}.
\begin{definition}[BSDEs with nonlinear weak terminal condition]\rm
Given a measurable map $\Psi: \textbf{R}\times \Omega \rightarrow U$, with $U \subset \textbf{A} \cup \{  -\infty\}$, $\textbf{A}$ a bounded subset of $\mathbf{R}$,  $\tau \in \mathcal{T}$, $\mu \in \textbf{L}_0(\textbf{R}, \mathcal{F}_\tau)$, a driver $f$  satisfying Assumption $\ref{f}$ and a measurable function $g$, we say that $(Y,Z) \in \mathbf{S}_2\times \mathbf{H}_2$ is a solution of the  BSDE $\left( f,\Psi,\mu,\tau,g \right)$ if
\begin{align}\label{Equ1}
\dd Y_t = Y_{T}+ \int_{t}^{T}g(s, Y_{s}, Z_{s})ds - \int_{t}^{T}Z_{s}dW_{s},\,\, 0 \leq t \leq T;
\end{align}
\begin{equation}\label{Equ2}
\cal{E}_{\tau,T}^{f}[\Psi(Y_{T})] \geq \mu.
\end{equation}
\end{definition}

Throughout the paper, we shall assume that the driver $g$ satisfies Assumption $\ref{f}$, with $C_g$ instead of $C_f$. To the coefficient $g$, we associate the nonlinear operator $\mathcal{E}^g$ defined as $\mathcal{E}^f,$ with $f$ replaced by $g$.\\
Let us now precise the hypothesis on the map $\Psi$ and the threshold $\mu$. We then discuss the wellposedness of the $BSDE(f,\Psi,\mu, \tau,g)$ under these assumptions. 

\begin{assumption}
For a.e. $\o\in  \O,$ the map $y \in \mathbf{R} \to \Psi(\o, y)$ is non-decreasing and valued in $[0; 1]\cup \{-\infty\}$ and 
its right-inverse $\Phi(\o, \cdot)$ is such that $\Phi:\O \times [0,1]\to [0,1]$ and it is measurable.
\end{assumption}

 This means that $\Psi(\omega,\cdot) \in [0,1]$ on $[0, \infty)$ and $\Psi(\omega,\cdot)=-\infty$ on $(-\infty,0)$. In view of the definition of the operator $\mathcal{E}^f,$ this implies that $Y_T \geq 0$ a.s.
Note that for notational simplicity we have considered the compact $[0,1]$, as in $\cite{BER}$, which can   be obviously replaced by an arbitrary compact set belonging to $\mathbf{R}$. Moreover, our analysis is the same if  for a.e. $\omega$ the map $\Psi(\omega,\cdot)$ is valued in   $[G_1(\omega), G_2(\omega)]$, with $G_1,G_2 \in \textbf{L}_2(\mathcal{F}_T).$

The threshold  $\mu$ is assumed to belong to $ \textbf{D}_\tau,$  where   $\textbf{D}_\tau$ corresponds to the set of random variables $ \{ \eta \in \textbf{L}_2(\mathcal{F}_\tau)\,\,\, \text{such that } \eta \in [\mathcal{E}^f_{\tau,T}[0], \mathcal{E}_{\tau,T}^f[1] ]\,  \text{ a.s.}\}.$We now introduce the following definition:

\begin{definition}\label{deff}

Let $f$ be a driver satisfying Assumption \ref{f}. For $i=1,2$, the solution of the BSDE associated to $f$ and $\xi^{i}$, with $\xi^{1}=0$ and $\xi^{2}=1$ is denoted by $(Y^{i}, Z^{i})$.
\end{definition}

Concerning the existence of a solution, remark that any random variable $\Phi(\xi)$, with $\xi \in [0,1]$ a.s. and $\mathcal{E}^f_{\tau,T}[\xi] \geq \mu$ could serve as terminal condition. However, the constraint is too weak to expect uniqueness.\\

We now introduce the value function $\mathcal{V}: \textbf{D} \rightarrow \textbf{L}_2$; $(\tau, \mu) \rightarrow \mathcal{V}(\tau,\mu)$, where $\textbf{D}:= \{(\tau, \mu); \,\, \tau \in \mathcal{T} \text{ and } \mu \in \textbf{D}_\tau  \}$ as follows:
\begin{align}\label{value!!}
\mathcal{V}(\tau,\mu):= \essinf \{Y_\tau: (Y,Z) \in \mathbf{S}_2 \times \textbf{H}_2  \text{ is a solution of BSDE} (f ,\Psi, \mu, \tau,g)\}.
\end{align}

The rest of the paper is dedicated to the study of the above map. In order to do it, we shall first establish the link with a control problem for BSDEs with strong terminal condition.

%

\subsubsection{Link with a control problem for BSDEs with strong terminal condition.}

In the spirit of   $\cite{BER}$ or \cite{BET}, we introduce an additional process $\mathcal{M}$ which allows to transform the weak constraint $\mathcal{E}_{0,T}^f[\Psi(Y_T)] \geq \mu$ into a strong one of the form $Y_T \geq \Phi(\mathcal{M}_T^{\mu})$. Since our constraint is expressed in terms of nonlinear BSDEs, the  process $\mathcal{M}$ is an $f$-martingale, contrary to  \cite{BER} and  \cite{BET} where $\mathcal{M}$ is a classical martingale.\\
\quad For each $\alpha \in \textbf{H}_2$, stopping time $\tau \in \mathcal{T}$ and $\mu \in \textbf{D}_\tau$, let $\mathcal{M}^{\tau,\mu, \alpha}$ be the $\mathbf{R}$-valued solution of the SDE:
\begin{align*}
\mathcal{M}_{t \vee \tau}^{\tau,\mu, \alpha}= \mu-\int_{\tau}^{t \vee \tau} f(s,\mathcal{M}_s^{\tau,\mu, \alpha},\alpha_s)ds+\int_\tau^{t \vee \tau} \alpha_s ^\top dW_s, \,\,\, 0 \leq t \leq T.
\end{align*}

%

We introduce the set of admissible controls $\textbf{A}_{\tau,\mu}$, which is defined as follows:
$$\textbf{A}_{\tau,\mu}:= \{\alpha \in \textbf{H}_2 \text{ such that  } \mathcal{M}^{\tau, \mu,\alpha} \in [Y^0,Y^1] \text{ on } \![\![\tau,T]\!]  \}.$$

Notice that for all $\alpha \in \textbf{A}_{\tau,\mu}$, $\Phi(\mathcal{M}_T^{\tau,\mu, \alpha})$ could serve as  terminal condition, since satisfies $\eqref{Equ2}.$ We thus introduce for all $\alpha \in \textbf{A}_{\tau,\mu}$   the BSDE with strong condtion $\Phi(\mathcal{M}_T^{\tau, \mu, \alpha})$ and driver $g$ and  define the value function $\mathcal{Y}(\tau,\mu)$ as follows:
\begin{equation}\label{probl}
\mathcal{Y}(\tau, \mu):= \essinf_{\alpha \in \textbf{A}_{\tau,\mu}} \mathcal{E}_{\tau,T}^g[\Phi(\mathcal{M}_T^{\tau, \mu,\alpha})].
\end{equation}

Our aim now is to link $\mathcal{Y}(\tau, \mu)$ to $\mathcal{V}(\tau,\mu)$, i.e. to prove that for all $\tau \in \mathcal{T}$ and  $\mu \in \textbf{D}_\tau$:
\begin{align}\label{erq}
\mathcal{V}(\tau,\mu)=\mathcal{Y}(\tau,\mu)\,\, \text { a.s.}
\end{align}
In order to explain the above equality between $\mathcal{V}$ and $\mathcal{Y}$, we state the following proposition:
\begin{proposition}\label{link}
Fix $\tau\in \cal{T},$ $\mu\in \textbf{D}_{\tau}$. Then $(Y,Z)\in \mathbf{S}_2\times \mathbf{H}_2$
is a solution of $BSDE (f,\Psi,\mu, \tau,g)$ if and only if $(Y,Z)$ satisfies \eqref{Equ1} and there exists
$\a \in \mb{A}_{\tau, \mu}$ such that $Y_{t}\geq \cal{E}_{t,T}^{g}[\Phi(\cal{M}_{T}^{\tau,\mu,\a})]$
for $t\in [0,T],$ $ \mathbb{P}$-a.s.
\end{proposition}
A sketch of proof is given in Appendix.

We come back to the explanation of equality \eqref{erq}.
\begin{itemize}
\item[$(i)$] Let $(Y,Z) \in \mathbf{S}_2 \times \textbf{H}_2$ be a solution of the BSDE$(f,\Psi,\mu,\tau,g)$. Then the above Proposition implies that it exists $\alpha \in \textbf{A}_{\tau,\mu}$ such that $Y_\tau \geq \mathcal{E}_{\tau,T}[\Phi(\mathcal{M}^{\tau,\mu,\alpha}_T)] \geq \mathcal{Y}(\tau,\mu),$ where the last inequality follows from  definition $\eqref{probl}$. By arbitrariness of $(Y,Z)$, we get $\mathcal{V}(\tau,\mu) \geq \mathcal{Y}(\tau,\mu)\,\, \text { a.s.}$
\item[$(ii)$] Fix $\alpha \in \textbf{A}_{\tau,\mu}$. Let $Z^\alpha$ be the associated process to the BSDE representation of $\Phi(\mathcal{M}^{\tau,\mu,\alpha}_T).$ Since $\Phi(\mathcal{M}^{\tau,\mu,\alpha}_T)$ is admissible as a  terminal condition, we obtain, by  Proposition \ref{link} that $(\mathcal{E}_{\cdot,T}[\Phi(\mathcal{M}^{\tau,\mu,\alpha})], Z^\alpha)$ is a solution, and thus $\mathcal{E}_{\tau,T}[\Phi(\mathcal{M}^{\tau,\mu,\alpha})] \geq \mathcal{V}(\tau,\mu).$ By arbitrariness of $\alpha$, we deduce
$\mathcal{V}(\tau,\mu) \leq \mathcal{Y}(\tau, \mu)\,\, \text { a.s.}$
\end{itemize}

From now on, we fix an initial condition $\mu_0 \in \textbf{D}_0$ at time $0$.  For each $\alpha \in \textbf{A}_{0, \mu_0}$ (denoted for simplicity $\textbf{A}_0$), we introduce the  process $(\mathcal{M}_t^{\alpha})_{t \leq T}$, representing a dynamic threshold controlled by the action of $\alpha$, which is defined as follows: 

$$\mathcal{M}_t^{\alpha}:= \mathcal{M}_t^{0,\mu_0, \alpha}.$$ 
We introduce for each $\tau \in \mathcal{T}$ the set of admissible controls coinciding with $\alpha$ up to the stopping  time $\tau$:
$$ \textbf{A}_{\tau}^{\alpha}:= \{\alpha' \in \textbf{A}_{\tau, \mathcal{M}_{\tau}^{\alpha}}: \alpha'=\alpha \, dt \otimes dP \,  \text{ on }  [\![0, \tau]\!]  \}.$$ 
The associated value is defined by:
$$\mathcal{Y}^{\alpha}(\tau){:=} \essinf_{\alpha' \in \textbf{A}_\tau^\alpha } \mathcal{E}_{\tau,T}^g[\Phi(\mathcal{M}_T^{\alpha'})](=\mathcal{Y}(\tau, \mathcal{M}_\tau^\alpha)).$$
In the following section, we shall investigate the time regularity of the above function and provide a BSDE representation. Before doing this, note that 
\begin{align}\label{B}
  |\mathcal{Y}^\alpha(\tau)|\leq \eta_\tau \text{ a.s. for all } \tau \in \mathcal{T}, 
\end{align}
$\text{where } \eta \text{ belongs to }  \mathbf{S}_2$  and it is given by  $\eta_t:=|\mathcal{E}_{t,T}^g[\Phi(1)]|+|\mathcal{E}_{t,T}^g[\Phi(0)]|,$ $t \leq T$.

\section{ Time regularity of the value function $\mathcal{Y}^\alpha$ and BSDE representation }

In this section, we study the regularity of the family $\{ \mathcal{Y}^\alpha(\tau), \tau \in \mathcal{T} \}$. More precisely, we show that it can be aggregated into a right continuous left limited process. The proof of this result becomes considerably more technical in our nonlinear case. Some comments regarding the main difficulties with respect to the case of linear constraints are provided in Remark \ref{comments}.

We first state the following dynamic programming principle.

\begin{lemma}\label{DPP}
For any $\a \in \textbf{A}_{0}$, ${\cal{Y}}^{\a}$ satisfies  the following dynamic programming principle: for all $\tau_1 \in \mathcal{T}$, $\tau_2 \in \mathcal{T}$ with $\tau_1 \leq \tau_2$ a.s. it holds:
\begin{align*}
{\cal{Y}}^{\a}({\tau_{1}})= \essinf_{\overline{\a}\in \textbf{A}_{\tau_1}^{\a}} \cal{E}_{\tau_{1},\tau_{2}}^{g}[\cal{Y}^{\overline{\a}}(\tau_{2})].
\end{align*}
\end{lemma}
Since the proof of the dynamic programming principle is based on classical arguments, we  refer the reader to \cite{BER}.\\
\quad We now make the following hypothesis on the map $\Phi$, under which we provide the time-regularity of our value function $\mathcal{Y}^\alpha$.
\begin{assumption}\label{continuity}
 The map  $m\in [0,1]\to \Phi(\o,m)$ is continuous for a.e. $\o\in \O$.
 \end{assumption}
 
 \begin{theorem} \label{cadlag}

Under the Assumption $\ref{continuity}$, for each $\alpha \in \textbf{A}_0$, there exists a right-continuous left limited process $(\overline{Y}^\alpha_t)_{t \leq T}$ which aggregates the family $\{\mathcal{Y}^\alpha(\tau), \,\, \tau \in \mathcal{T}\}.$
\end{theorem}

\begin{proof}

By Lemma $\ref{DPP}$, we easily obtain that the family $\{ -\mathcal{Y}^\alpha(\tau),\,\, \tau \in \mathcal{T}\}$ is a $-g(-)$ supermartingale system. Since moreover $\eqref{B}$ holds, we can apply Lemma A.2 in \cite{BPT} and obtain the existence of an optional ladlag process, denoted by $(\mathcal{Y}^\alpha_t)_{t \leq T}$ which aggregates the family, that is $ \mathcal{Y}^\alpha(\tau)=\mathcal{Y}^\alpha_\tau$, for all $\tau \in \mathcal{T}$. Hence, the following limits:
$$
\lim\limits_{s\in  (t,T] \downarrow t}\mathcal{Y}^\alpha_s \;\; \text{and }\;\; \lim\limits_{s \in (t,T] \uparrow t}\mathcal{Y}^\alpha_{s}.
$$
are well-defined and finite.\\
Now, we define:
\begin{align}\label{right}
\overline{\mathcal{Y}}^\alpha_{t}: = \lim\limits_{s \in (t,T] \downarrow t}\mathcal{Y}^\alpha_{s},\; t \in [0,T[, \,\,\, \overline{\mathcal{Y}}^\alpha_{T}:=\mathcal{Y}^\alpha_{T}.
\end{align}
which is by definition a real-valued RCLL process.

In order to prove the desired regularity property, we have to show that for
every stopping time $\tau \in \mathcal{T}$, it holds that:
$$
\overline{\mathcal{Y}}^\alpha_{\tau}=\mathcal{Y}^\alpha_{\tau}\,\,\, a.s.
$$
The above relation implies that the processes $\overline{\mathcal{Y}}^\alpha$ and $\mathcal{Y}^\alpha$ are indistinguishable. The proof is divided in two steps.\\

\textbf{Step 1.} Fix $\tau \in \mathcal{T}$. We first prove that $\overline{\mathcal{Y}}^\alpha_\tau \leq \mathcal{Y}^\alpha_\tau$ a.s.\\
\textbf{a.} Let $\a^{\prime} \in \mb{A}_{\tau}^{\a}$. Fix $k \in \mathbb{N}^*$.\\
Define $\tilde{\cal{M}}_{T}^{k,\a^{\prime}}: = \frac{1}{k}+\cal{M}_{T}^{\a^{\prime}}(1-\frac{1}{k})$. Note that  $\tilde{\cal{M}}_{T}^{k,\a^{\prime}} \geq \cal{M}_{T}^{\a^{\prime}}$ and $\tilde{\cal{M}}_{T}^{k,\a^{\prime}} \rightarrow \cal{M}_{T}^{\a^{\prime}}$ when $k \rightarrow \infty.$ 
 In the sequel, we denote by $(\cal{E}_{\cdot,T}^{f}[\tilde{\cal{M}}_{T}^{k,\a^{\prime}}], \tilde{Z}^k)$  the solution of the BSDE associated
to $(\tilde{\cal{M}}_{T}^{k,\a^{\prime}},f)$. \\
Recall that $\cal{M}_{T}^{\a^{\prime}}(\omega)$ belongs for a.e. $\omega$ to $[0,1]$. Hence, by construction, we have:
$$
0 \leq \cal{M}_{T}^{\a^{\prime}} \leq \tilde{\cal{M}}_{T}^{k,\a^{\prime}} \leq 1 \, \, \, \text{a.s.}
$$
By applying the comparison theorem for BSDEs  and since $\alpha' \in \textbf{A}_\tau^{\alpha}$, we obtain:
\begin{align}\label{iineq}
\mathcal{E}_{\tau,T}^f[0] \leq \cal{M}_{\tau}^{\a} \leq \mathcal{E}_{\tau,T}^f[\tilde{\cal{M}}_{T}^{k,\a^{\prime}}] \, \, \, \text{a.s.}
\end{align}

We claim that it exists a sequence of stopping times $ (\tau_{n,k})_n$ valued a.s. in $[0,T]$ and an admissible control $\Tilde{\alpha}_{k} \in \textbf{A}_{ \tau_{n,k}}^\alpha$ for all $n \in \mathbb{N}$ such that: $ \tau_{n,k}$  $ \rightarrow \tau$ when $n$ tends to $+\infty$,   $ \tau_{n,k} >\tau$ a.s. on $\{ \tau <T \}$ for all $n \in \mathbb{N}$ and   $\mathcal{M}_T^{\Tilde{\alpha}_{k}} \leq \Tilde{\mathcal{M}}_T^{k,\alpha'}$. The proof is postponed to Step 1.b.\\
Thanks to the above assertion, we  can appeal to $\eqref{right}$ and obtain:
\begin{align}\label{ineqqqq}
\overline{\mathcal{Y}}^\alpha_{\tau} = \lim\limits_{n\to \infty}\mathcal{Y}^\alpha_{\tau_{n,k}} \,\,\, \text{a.s.}
\end{align}
Using the definition of $\mathcal{Y}^\alpha$, we get:
\begin{align}\label{ineqqqqq}
\mathcal{Y}^\alpha_{\tau_{n,k}} \leq \cal{E}_{\tau_{n,k},T}^{g}[\Phi(\cal{M}_{T}^{\Tilde{\a}_{k}})] \, \, \, \text{a.s.}
\end{align}
As $\cal{M}_{T}^{\Tilde{\a}_{k}} \leq \tilde{\cal{M}}_{T}^{k,\a^{\prime}}$ a.s. and $\Phi$ is nondecreasing, by applying  the comparison theorem for BSDEs, we get for all $n$:
\begin{align*}
\cal{E}_{\tau_{n,k},T}^{g}[\Phi(\cal{M}_{T}^{\Tilde{\a}_{k}})] \leq \cal{E}_{\tau_{n,k},T}^{g}[\Phi(\tilde{\cal{M}}_{T}^{k,\a^{\prime}})] \,\,\, \text{ a.s. }
\end{align*}
The above inequality together with \eqref{ineqqqq}, \eqref{ineqqqqq} and the continuity of the process $\mathcal{E}^f_{\cdot,T}[\Phi(\tilde{\cal{M}}_{T}^{k,\a^{\prime}})]$ lead to:
\begin{align*}
\overline{\mathcal{Y}}^\alpha_{\tau}  \leq \cal{E}_{\tau,T}^{g}[\Phi(\tilde{\cal{M}}_{T}^{k,\a^{\prime}})] \,\,\, \text{ a.s. }
\end{align*}
Since $\tilde{\cal{M}}_{T}^{k,\a^{\prime}} \rightarrow  \cal{M}_{T}^{\a^{\prime}}$ a.s. and $\Phi$ is a.s. continuous,  by letting $k$ tend to $\infty$, we obtain:
$$
\overline{\mathcal{Y}}^\alpha_{\tau}\leq \cal{E}_{\tau,T}^{g}[\Phi(\cal{M}_{T}^{\a^{\prime}})] \,\, \text{a.s}.
$$
By arbitrariness of $\a^{\prime} \in \textbf{A}_\tau^{\alpha}$, we  conclude:
$$
\overline{\mathcal{Y}}^\alpha_{\tau}\leq \mathcal{Y}^\alpha_\tau \,\, \text{a.s}.
$$
\textbf{b.} i) We first construct, for each $k \in \mathbb{N}^*$, the sequence of stopping times $(\tau_{n,k})_{n}$ valued a.s. in $[0,T]$ such that $\tau_{n,k} \rightarrow \tau$ when $n \rightarrow \infty$ and $\tau_{n,k}>\tau$ a.s. on $\{\tau>T\}$ for all $n \in \mathbb{N}$.\\ 
To do this, we start by defining the following stopping time:
\begin{align}
&\sigma_{k}: = \inf \{\tau \leq t \leq T; \cal{M}_{t}^{\a} = \cal{E}_{t,T}^{f}[\tilde{\cal{M}}_{T}^{k,\a^{\prime}}]\}.
\end{align}
We use the convention $\inf \emptyset=+\infty$.

We introduce $(\tau_n)_n$ a sequence of stopping times with values in $[0,T]$ such that  $\tau_n > \tau$ on $\{ \tau <T \}$ for all $n$ and $\tau_n \rightarrow \tau$ a.s. when $n$ tends to $+ \infty.$\\
For each $n$, we define $\tau_{n,k}$ as follows:
\begin{equation}\label{stoptime}
\tau_{n,k}:=\tau_n \textbf{1}_{A_k}+\left(\tau_n \wedge  {\sigma}_{k} \right)\textbf{1}_{A_k^c},
\end{equation}
with
$$A_k:= \{\mathcal{E}_{\tau,T}^f[\Tilde{\cal{M}}_{T}^{k,\alpha^{\prime}}] - \cal{M}_{\tau}^{\a}=0 \} \in \mathcal{F}_\tau;\,\,\, A_k^c:= \{ \mathcal{E}_{\tau,T}^f[\Tilde{\cal{M}}_{T}^{k,\alpha^{\prime}}] - \cal{M}_{\tau}^{\a}>0 \} \in \mathcal{F}_\tau.$$
Remark that by $\eqref{iineq}$, $P(A_k \cup A_k^c)=1$ and thus $\tau_{n,k} \downarrow \tau$ a.s. when $n \rightarrow \infty.$ 
We precise that we have to introduce the sets $A_k$ and $A_k^c$ because $\sigma_k=\tau$  on $A_k$. In order to have $\tau_{n,k}>\tau$ a.s. on $\{\tau>T\}$, it  remains to prove that $\tau < \sigma_k$ on $A_k^c$.

The definition of $\sigma_k$ together with the continuity of the processes $\cal{M}^{\a}$ and  $\mathcal{E}_{\cdot,T}^f[\Tilde{\cal{M}}^{k,\a^{\prime}}]$, imply that almost surely, $\sigma_k=+ \infty$ or $\mathcal{E}_{\sigma_k,T}^f[\Tilde{\cal{M}}^{k,\a^{\prime}}_{T}] \leq \cal{M}^{\a}_{\sigma_k}$. Moreover, since on $A_k^c$ we have $\mathcal{E}^f_{\tau,T}[\Tilde{\cal{M}}_{T}^{k,\alpha^{\prime}}] > \cal{M}_{\tau}^{\a}$ and  $\tau \leq \sigma_k$ a.s., one can thus conclude that
\begin{align*}
\tau< \sigma_k  \text{ } \text{a.s. on } A_k^c.
\end{align*}
ii) We provide the existence of an admissible control  $\Tilde{\alpha}_{k} \in \textbf{A}_{\tau_{n,k}}^\alpha$ for all $n \in \mathbb{N}$ such that $\mathcal{M}_T^{\Tilde{\alpha}_{k}} \leq \Tilde{\mathcal{M}}_T^{k, \alpha'}$.The control $\Tilde{\a}_{k}$ is defined as follows:
\begin{equation*}
\Tilde{\a}_{k}:= \a_s \mb{1}_{ \{ s \leq  \Tilde{\sigma}_{k} \}} + \Tilde{Z}^k_s \textbf{1}_{ \{ s>  \Tilde{\sigma}_{k}\}},
\end{equation*}
where $\Tilde{\sigma}_k=\sigma_k \wedge T.$ Recall that $\Tilde{Z}^k$ is the process associated to the BSDE representation of $\tilde{\cal{M}}_{T}^{k,\a^{\prime}}$.\\
Note that the above construction ensures that $0 \leq \cal{M}_{T}^{\Tilde{\a}_{k}} \leq \tilde{\cal{M}}_{T}^{k,\a^{\prime}}$ a.s. It remains to show that $\Tilde{\alpha}^{k} \in \textbf{A}^{\alpha}_{\tau_{n,k}}.$
 It is clear that we have:
 \begin{equation*}
 \mathcal{M}_{\tau_n \wedge  \sigma_{k}}^{\alpha}=\mathcal{M}_{\tau_n \wedge   \sigma_{k}}^{\Tilde{\alpha}^{k}}\,\,\, \text{ a.s.}
 \end{equation*}
 and hence
 \begin{align}\label{MM1}
 \mathcal{M}_{\tau_n \wedge  \sigma_{k}}^{\alpha}=\mathcal{M}_{\tau_n \wedge   \sigma_{k}}^{\Tilde{\alpha}^{k}}\,\,\, \text{ a.s. on } A_k^c.
 \end{align}

Since $\sigma_k = \tau $ on $A_k$, we have to prove that $\mathcal{M}_{\tau_n}^{\alpha}=\mathcal{M}_{\tau_n}^{\Tilde{\alpha}^{k}} \,\,\, \text{ a.s. on } A_k.$
Recall that $\alpha' \in \textbf{A}_\tau^{\alpha}$. Hence, by definition of the set $A_k$, we obtain $\mathcal{M}_\tau^{\alpha'}=\mathcal{E}_{\tau,T}^f[\Tilde{\mathcal{M}}_T^{k, \alpha'}]$ a.s. on $A_k$. A strict comparison theorem for BSDEs and the definition of $\Tilde{\mathcal{M}}_T^{k, \alpha'}$ lead to
\begin{align}\label{first}
\Tilde{\mathcal{M}}_T^{k, \alpha'}=\mathcal{M}_T^{\alpha'}=1 \,\,\,  \text{ a.s. on } A_k.
\end{align}

By uniqueness of the solution of a BSDE, we get:
\begin{align*}
\mathcal{E}_{\tau_n,T}^f[\Tilde{\mathcal{M}}_T^{k, \alpha'}]=\mathcal{M}_{\tau_n}^{\alpha'}=\mathcal{E}_{\tau_n,T}^f[1] \,\,\, \text{ a.s. on } A_k.
\end{align*}

Moreover, by $\eqref{first}$ and the comparison theorem for BSDEs, we have $\mathcal{M}_\tau^{\alpha'}=\mathcal{E}_{\tau,T}^f[1]$ a.s. on $A_k$ and since $\alpha' \in \textbf{A}_{\tau}^{\alpha}$, we get $\mathcal{M}_\tau^{\alpha}=\mathcal{E}_{\tau,T}^f[1]$ a.s. on $A_k$. The strict comparison theorem for BSDEs allows us to conclude that:
\begin{align*}
\mathcal{M}_{\tau_n}^{\alpha}=\mathcal{E}_{\tau_n,T}^f[1] \text{ a.s.  on } A_k.
\end{align*}
The two above equalities imply:

\begin{align}\label{second}
\mathcal{M}_{\tau_n}^{\alpha}= \mathcal{E}_{\tau_n,T}^f[\Tilde{\mathcal{M}}_T^{k, \alpha'}] \,\, \text{ a.s. on} A_k.
\end{align}

Now, recall that $\mathcal{E}_{\tau,T}^f[\Tilde{\cal{M}}_{T}^{k,\alpha^{\prime}}]=\cal{M}_{\tau}^{\a}$ on $A_k$. The definition of the control $\Tilde{\alpha}^k$ together with the fact that $\sigma_k=\tau$  on $A_k$ lead to:

\begin{align}\label{third}
\mathcal{E}_{\tau_n,T}^f[\Tilde{\mathcal{M}}_T^{k, \alpha'}]=\mathcal{M}_{\tau_n}^{\Tilde{\alpha}^{k}}  \text{ a.s.  on } A_k.
\end{align}

Using $\eqref{second}$,  $\eqref{third}$ hold, we finally obtain:
\begin{align}\label{MM2}
\mathcal{M}_{\tau_n}^{\alpha}=\mathcal{M}_{\tau_n}^{\Tilde{\alpha}^{k}}  \text{ a.s.  on } A_k.
\end{align}

By \eqref{MM2} we deduce that $\Tilde{\alpha}^{k} \in \textbf{A}_{\tau_{n,k}}^{\alpha}.$

\textbf{Step 2.}
Let us prove now the converse inequality $\overline{\mathcal{Y}}^\alpha_{\tau}\geq \mathcal{Y}^\alpha_{\tau} \,\, \text{a.s}.$\\
By applying on $[\tau, \tau_n]$ the stability result  for  BSDEs with  parameters $(\overline{\mathcal{Y}}^\alpha_\tau, 0)$ and ($\mathcal{Y}_{\tau_n}^{\alpha}, g\textbf{1}_{[0, \tau_n)}$), we obtain:
\begin{equation}\label{ineq}
|\!|\overline{\mathcal{Y}}^\alpha_\tau-\mathcal{E}_{\tau,\tau_n}^g[\mathcal{Y}^{\alpha}_{\tau_n}] |\!|_{\textbf{L}_2} \leq C \left(|\!|\overline{\mathcal{Y}}^\alpha_\tau-{\mathcal{Y}}_{\tau_n}^{\alpha} |\!|_{\textbf{L}_2}+\textbf{E}[\int_{\tau}^{\tau^n}|g(s,\overline{\mathcal{Y}}^\alpha_\tau,0)|^2ds]\right).
\end{equation}
 The assumptions on the driver $g$, the convergence of $\tau_n$ to $\tau$, the integrability of $\mathcal{Y}^\alpha$ (see \eqref{B}), and  Lebesgue's Theorem imply that $\textbf{E}[\int_{\tau}^{\tau^n}|g(s,\overline{\mathcal{Y}}^\alpha_\tau,0)|^2ds] \rightarrow 0$. By the same arguments and \eqref{right}, we get $|\!|\overline{\mathcal{Y}}^\alpha_\tau-{\mathcal{Y}}_{\tau_n}^{\alpha} |\!|_{\textbf{L}_2} \rightarrow 0$. Now, we let $n$ tend to $\infty$ in \eqref{ineq}, and  obtain $\mathcal{E}_{\tau,\tau_n}^g[\mathcal{Y}^{\alpha}_{\tau_n}] \rightarrow  \overline{\mathcal{Y}}^\alpha_\tau$ a.s., up to a subsequence.

Moreover, Lemma  \ref{DPP} implies that  $\mathcal{E}_{\tau,\tau_n}^g[\mathcal{Y}_{\tau_n}^{\alpha}] \geq \mathcal{Y}_{\tau}^{\alpha}.$ This inequality and the above convergence lead to the desired result.
\end{proof}

\begin{remark}\label{comments}
In \cite{BER}, it is provided  the existence of a control $\alpha_n \in \textbf{A}_{\tau_n}^\alpha$, with $\tau_n \rightarrow \tau$ and $\tau_n>\tau$ for all n, such that $M_T^{\alpha_n}$ remains ''sufficiently close'' to $M_T^{\alpha'}$. The control $\alpha_n$ is obtained by scaling $\alpha$  in  an appropriate way. This approach cannot be applied in the case of nonlinear constraints, as being clearly specific to the linear setting. 
\end{remark}

Using similar arguments as in Theorem 2.1 in \cite{BER} (points (iii), (iv)) one can show the following BSDE representation for $\mathcal{Y}^\alpha$:

\begin{theorem}
Assume that Assumption $\ref{continuity}$ holds. Then  there exists a family $(\cal{Z}^{\a}, \cal{K}^{\a})_{ \a \in \textbf{A}_{0}}$ satisfying
\begin{equation}\label{Eq1}
\sup_{\a \in \textbf{A}_{0}}\|\cal{Y}^{\a}, \cal{Z}^{\a}, \cal{K}^{\a}\|_{\mathbf{S}_2 \times \textbf{H}_2 \times \textbf{I}_2}<+\infty.
\end{equation}
and such that for all $\a \in \textbf{A}_{0}$, we have
\begin{equation}\label{Eqqqq2}
\cal{Y}_t^{\alpha}=\Phi (\cal{M}_{T}^{\alpha}) + \int_{t}^{T}g(s, \cal{Y}_{s}^{\a}, \cal{Z}_{s}^{\a})ds-\int_{t}^{T}\cal{Z}_{s}^{\a}dW_{s} +\cal{K}_t^{\a}-\cal{K}_{T}^{\a}.
\end{equation}
\begin{equation}\label{Eqq3}
\cal{K}_{\tau_1}^{\alpha}= \essinf_{\bar{\alpha} \in \textbf{A}_{\tau_1}^{\alpha}} \textbf{E}[\cal{K}_{\tau_2}^{\bar{\alpha}}| \mathcal{F}_{\tau_1}], \,\, \forall \tau_1 \in \mathcal{T}, \,\, \tau_2 \in \mathcal{T}_{\tau_1},
\end{equation}
and
\begin{equation}\label{Eq4}
(\mathcal{Y}^{\alpha}, \mathcal{Z}^{\alpha}, \mathcal{K}^{\alpha})\text{\textbf{1}}_{[0,\tau]}=(\mathcal{Y}^{\alpha}, \mathcal{Z}^{\alpha}, \mathcal{K}^{\alpha})\text{\textbf{1}}_{[0,\tau]}, \,\, \forall \tau \in \mathcal{T}, \,\, \bar{\alpha} \in \textbf{A}_\tau^{\alpha}.
\end{equation}

Moreover, $(\cal{Y}^{\a}, \cal{Z}^{\a}, \cal{K}^{\a})_{\a \in \textbf{A}_{0}}$ is the unique family satisfying \eqref{Eq1}, \eqref{Eqqqq2}, \eqref{Eqq3} and \eqref{Eq4}.
\end{theorem}

\section{Existence of optimal controls in the case of concave constraints.}

We show that in the case of concave constraints and under convexity assumptions on $\Phi$ and $g$, we get the existence of an optimal control $\hat \alpha$, that is  $\mathcal{Y}_t^{\hat \alpha}=\mathcal{E}_{t,T}^{g}[\Phi(\mathcal{M}_T^{\hat{\alpha}})].$ 

For all $(\lambda, m_1, m_2,t,y_1,y_2,z_1,z_2) \in [0,1] \times [0,1]^2 \times [0,T] \times \mathbf{R}^2 \times [\mathbf{R}^{d}]^2$, we assume a.s. the following:\\
$\mathbf{(H_{conc})}$
$$ \lambda f(t,y_1,z_1)+(1-\lambda)f(t,y_2,z_2) \leq f(t,\lambda y_1+(1-\lambda)y_2,\lambda z_1+(1-\lambda)z_2).$$
$\mathbf{(H_{conv})}$
$$\Phi(\lambda m_1+(1-\lambda)m_2) \leq \lambda \Phi(m_1)+(1-\lambda)\Phi(m_2)$$
$$g(t,\lambda y_1+(1-\lambda)y_2,\lambda z_1+(1-\lambda)z_2) \leq \lambda g(t,y_1,z_1)+(1-\lambda)g(t,y_2,z_2).$$

\begin{proposition}\label{optimal}

Under Hypothesis $\mathbf{(H_{conv})}$  and $\mathbf{(H_{conc})}$, for any $(\tau, \alpha) \in \mathcal{T} \times \textbf{H}_2$, there exists $\hat{\alpha}^{\tau,\alpha} \in \textbf{A}_\tau^{\alpha}$ such that
$$\mathcal{Y}_\tau^{\alpha}=\mathcal{E}_\tau^g \left[\Phi(\mathcal{M}_T^{\hat{\alpha}^{\tau,\alpha}})  \right]=\mathcal{E}_{\tau,\tau'}^g \left[ \mathcal{Y}_{\tau'}^{\hat{\alpha}^{\tau,\alpha}} \right], \,\,\, \forall \tau' \in \mathcal{T}_{\tau}.$$
\end{proposition}

\begin{proof}
By Lemma \ref{sequence} in the Appendix, there exists $(\alpha^n)_n \in \textbf{A}_\tau^\alpha$ such that:
\begin{align}\label{convvvv}
\mathcal{Y}_\tau^\alpha= \lim_{n \rightarrow \infty} \mathcal{E}_{\tau,T}^g[\Phi(\mathcal{M}_T^{\alpha^n})].
\end{align}
Recall that $(\mathcal{M}_T^{\alpha^n})_n$ is valued in [0,1]. By Komlos Theorem, $\Tilde{\mathcal{M}}_T^{n}:= \frac{1}{n} \sum_{i \leq n }  \mathcal{M}_T^{\alpha^{i}}$ converges  a.s. to a random variable  $\Tilde{\mathcal{M}_T}$ which belongs  a.s. to $[0,1]$.\\
From the concavity assumption on the driver $f$ and  the comparison theorem for BSDEs we get:
\begin{align}\label{m1}
\mathcal{E}_{\tau,T}^{f}[\Tilde{\mathcal{M}}_T^{n}] \geq \dfrac{1}{n} \sum_{i \leq n } \mathcal{E}_{\tau,T}^f [\mathcal{M}_T^{\alpha^{i}}] = \mathcal{M}_{\tau}^{\alpha},
\end{align}
since $\alpha^n \in \textbf{A}_\tau^\alpha$  for all $n$.\\
The a priori estimates for BSDEs lead to:
\begin{align*}
|\mathcal{E}_{\tau,T}^{f}[\Tilde{\mathcal{M}}_T^{n}]-\mathcal{E}_{\tau,T}^{f}[\Tilde{\mathcal{M}}_T] |^2 \leq \textbf{E}_t[|\Tilde{\mathcal{M}}_T^{n}-\Tilde{\mathcal{M}}_T|^2].
\end{align*}
The a.s. convergence $\Tilde{\mathcal{M}}_T^{n} \rightarrow \Tilde{\mathcal{M}_T}$ and the boundness of the sequence $(\Tilde{\mathcal{M}}_T^{n})_n$ allow us to apply the Lebesgue's theorem and to derive that the right hand side of the above inequality tends to $0$ when $n$ goes to $+ \infty.$ We thus derive that:
\begin{align}\label{m2}
\mathcal{E}_{\tau,T}^{f}[\Tilde{\mathcal{M}}_T^{n}] \rightarrow \mathcal{E}_{\tau,T}^{f}[\Tilde{\mathcal{M}}_T] \,\,\,{a.s.}
\end{align}
Hence, inequality \eqref{m1} combined with \eqref{m2} lead to
$\mathcal{E}_{\tau,T}^{f}[\Tilde{\mathcal{M}}_T] \geq \mathcal{M}_{\tau}^{\alpha}.$

Let us denote by $\Tilde{\alpha}$ the  control associated to the BSDE  with terminal condition $\Tilde{\mathcal{M}}_T$ and driver $f$.
We define the following stopping time:
$$
\theta^{\tilde{\a}}:=\inf \{ \tau \leq s \leq T: \cal{M}_{s}^{\tau, \mathcal{M}_{\tau}^{\alpha},\tilde{\a}} = \cal{E}_{s,T}^{f}[0]\} \wedge T,
$$
with the convention $\inf \emptyset =+\infty.$
We recall that  $(Y^{0},Z^{0})$ represents the solution of the BSDE associated to driver $f$ and terminal condition
$0$ and we define the control $\hat{\a}$ as follows:
\begin{align*}
\hat{\alpha}_{s}: = \alpha_{s} \textbf{1}_{s \leq \tau}+ \tilde{\a}_{s}\mb{1}_{\{\tau < s\leq \t^{\tilde{\a}}\}} + Z_{s}^{0}\mb{1}_{\{s>\t^{\tilde{\a}}\}}.
\end{align*}

Note that $\hat{\a}$ belongs to $\textbf{A}_{\tau}^{\alpha}$. Moreover, by construction, we have:
\begin{align} \label{ee1}
\cal{M}_{T}^{\hat{\alpha}} \leq \Tilde{\cal{M}}_{T} \,\,\, \text{a.s.}
\end{align}
Now, by using hypothesis $\mathbf{(H_{conv})}$ and the comparison theorem, we obtain:
\begin{align}\label{BV}
\Tilde{\mathcal{Y}}_\tau^n:=  \dfrac{1}{n} \sum_{i \leq n }  \mathcal{E}_{\tau,T}^g \left[\Phi(\mathcal{M}_T^{\alpha^{i}})  \right] \geq \mathcal{E}_{\tau,T}^{g} \left[\Phi(\Tilde{\mathcal{M}}_T^{n}) \right].
\end{align}
By $\eqref{convvvv}$ and Cesaro's Lemma we have $\lim_{n \rightarrow \infty} \Tilde{\mathcal{Y}}_\tau^n=\mathcal{Y}_\tau^\alpha$ a.s.

Similar arguments as the ones used to prove $\eqref{m2}$ allow us to deduce that $\lim_{n \rightarrow \infty} \mathcal{E}_{\tau,T}^{g} \left[\Phi(\Tilde{\mathcal{M}}_T^{n})\right]=\mathcal{E}_{\tau,T}^{g} \left[\Phi(\Tilde{\mathcal{M}}_T)\right]$a.s. By letting $n$ tend to $\infty$ in \eqref{BV} we conclude:
\begin{align} \label{ee2}
\mathcal{Y}_\tau^{\alpha} \geq \mathcal{E}_{\tau,T}^{g} \left [ \Phi(\Tilde{\mathcal{M}}_T) \right].
\end{align}
From \eqref{ee1}, \eqref{ee2}, the non-decreasing monotonicity of the map $\Phi$ and the comparison theorem for BSDEs, we finally get:
\begin{align}
\mathcal{Y}_\tau^{\alpha} \geq \mathcal{E}_{\tau,T}^{g} \left [ \Phi(\cal{M}_{T}^{\hat{\alpha}}) \right].
\end{align}
The equality follows by definition of  $\mathcal{Y}_\tau^{\alpha}$ and $\hat{\alpha}$ is hence the optimal control. \\
In order to show the second equality $\mathcal{Y}_\tau^{\alpha}=\mathcal{E}_{\tau,\tau'}^g \left[ \mathcal{Y}_{\tau'}^{\hat{\alpha}^{\tau,\alpha}} \right], \,\,\, \forall \tau' \in \mathcal{T}_{\tau}$, we first observe that $\mathcal{Y}_\tau^{\alpha}=\mathcal{E}_{\tau,\tau'}^g \left[ \mathcal{E}_{\tau',T}^g[\Phi(\mathcal{M}_T^{\hat{\alpha}})] \right] \geq \mathcal{E}_{\tau,\tau'}^g \left[\mathcal{Y}_{\tau'}^{\hat{\alpha}}\right],$ by definition of the value function $\mathcal{Y}_{\tau'}^{\hat{\alpha}}$ and the comparison theorem. As above, there exists $(\hat{\alpha}^n) \in \textbf{A}_{\tau'}^{\hat{\alpha}}$ such that $\mathcal{E}_{\tau',T}^g[\Phi(\mathcal{M}_T^{\hat{\alpha}^n})] \rightarrow \mathcal{Y}_{\tau'}^{\hat{\alpha}}$ a.s. By \eqref{B}, the convergence also holds  in $\textbf{L}_2$. The a priori estimates on BSDEs give: $\mathcal{Y}_\tau^{\alpha} \leq \mathcal{E}_{\tau,\tau'}^g \left[ \mathcal{E}_{\tau',T}^g[\Phi(\mathcal{M}_T^{\hat{\alpha}^n})] \right] \rightarrow \mathcal{E}_{\tau,\tau'}^g[\mathcal{Y}_{\tau'}^{\hat{\alpha}}].$
\end{proof}

\begin{remark}
Note that in $\cite{BER}$, the optimal control is obtained directly by using the martingale representation of $\Tilde{\mathcal{M}}_T$,  due to the linearity of the expectation. In our nonlinear case, that is no longer possible and we need a more complicated construction of the optimal control.
\end{remark}

\section{Properties of the value function}

In this section, we study the continuity and the convexity (defined in a probabilistic sense) of the map $\mathcal{Y}_t(\mu):=\mathcal{Y}(t,\mu)$ with respect to the threshold $\mu$, for any $t <T$.

%
%


\subsection{Continuity}
Fix $t \in [0,T]$. We give below an estimate on the map $\mu \rightarrow \mathcal{Y}_t(\mu)$, ensuring its continuity under some weak assumptions on the map $\Phi$ ( e.g. $\Phi$ is Lipschitz continuous with respect to $x$, uniformly in $\omega$ or deterministic continuous). We obtain a more natural  bound for $|\mathcal{Y}_t(\mu_1)-\mathcal{Y}_t(\mu_2)|$ than the one provided in the case of classical expectations constraints ( see \cite{BER}), which is expressed only through the spread $|\mu_1-\mu_2|^{\frac{1}{2}}$ ( in $\cite{BER}$  it depends on $(1-\frac{\mu_1}{\mu_2})\textbf{1}_{\mu_1<\mu_2}+\frac{\mu_1-\mu_2}{1-\mu_2}\textbf{1}_{\mu_1>\mu_2}$; $(1-\frac{\mu_2}{\mu_1})\textbf{1}_{\mu_2<\mu_1}+\frac{\mu_2-\mu_1}{1-\mu_1}\textbf{1}_{\mu_1<\mu_2}$ and on other two terms related to the case when the thresholds take the boundary values $0$ and $1$). Moreover, our proof is based on BSDEs techniques,  allowing to treat the nonlinear case, contrary to \cite{BER}, where the arguments hold only in the case of linear constraints.\\

\begin{theorem}\label{cont}
Let $t < T$, $ \text {and } \mu_{1},  \mu_{2}\in \textbf{D}_t$.\\
Then $|\mathcal{Y}_t(\mu_1)-\mathcal{Y}_t(\mu_2)| \leq Err_t \left(  \Delta(\mu_1, \mu_2)\right)$, where $\Delta(\mu_1, \mu_2)=C|\mu_1-\mu_2|^{\frac{1}{2}},$ with $C$ a constant depending only on $(C_f,T)$ and
\begin{align}
Err_t(\xi):=  \esssup \{\mathcal{R}_t(M,M'): M,M' \in \textbf{L}_0([0,1]), \textbf{E}_t[|M-M'|^2] \leq \xi  \},
\end{align}
where $\xi \in \textbf{L}_2(\mathbf{R}, \mathcal{F}_t)$ and $\mathcal{R}_t(M,M'):=|\mathcal{E}_{t,T}^g[\Phi(M)]-\mathcal{E}_{t,T}^g[\Phi(M')]|.$
\end{theorem}

\begin{proof}
We define $\Tilde \mu_1:=\mu_1 \vee \mu_2$ and $\Tilde{\mu_2}:= \mu_1 \wedge \mu_2.$
By the monotonocity property of the map $\mu \rightarrow \mathcal{Y}_t(\mu)$ (\ref{monotonicity}), we have $\mathcal{Y}_t(\Tilde{\mu_1}) \geq \mathcal{Y}_t(\Tilde{\mu_2})$ a.s.\\
By Lemma \ref{sequence}, it exists $\alpha^n \in \textbf{A}_{t,\Tilde{\mu}_2}$ such that $\lim_{n \rightarrow \infty}\mathcal{E}_{t,T}^g[\Phi(\mathcal{M}^{\Tilde{\mu}_2,\alpha^n})]=\mathcal{Y}_t(\Tilde{\mu}_2)$ a.s.
Fix $n \in \mathbb{N}$. We now construct an admissible control  $\Tilde{\alpha}^n \in \textbf{A}_{t,\Tilde{\mu}_1} $ such that $\mathcal{M}^{\Tilde{\mu}_1, \Tilde{\alpha}^n}_s \in \left[\mathcal{M}^{\Tilde{\mu}_2, \alpha^n}_s, \mathcal{E}^f_{s,T}[1]\right]$, $t \leq s \leq T,$ a.s. It is defined as follows:
$$\Tilde{\a}^{n}_s:= \a_s^n \mb{1}_{ \{ s \leq \tau \}} + Z_s^1 \textbf{1}_{ \{ s >\tau\}},$$
where  $\tau:= \inf \{ s \in [t,T]: \cal{M}_{s}^{\Tilde{\mu}_{1},\a^{n}} = \cal{E}_{s}^{f}[1]\} \wedge T$, with the convention $\inf \emptyset=+\infty$. Recall that $Z^1$ corresponds to the control associated to the BSDE of terminal condition $1$ and driver $f$.

By definition of the value function $\mathcal{Y}_t$, we get:
\begin{align}\label{eqp}
\mathcal{Y}_t(\Tilde{\mu_1}) \leq \mathcal{E}_{t,T}^g[\Phi(\mathcal{M}_T^{\Tilde{\mu}_1, \Tilde{\alpha}^n})]=\mathcal{E}_{t,T}^g[\Phi(\mathcal{M}_T^{\Tilde{\mu}_1, \Tilde{\alpha}^n})]-\mathcal{E}_{t,T}^g[\Phi(\mathcal{M}_T^{\Tilde{\mu}_2, \alpha^n})]+\mathcal{E}_{t,T}^g[\Phi(\mathcal{M}_T^{\Tilde{\mu}_2, \alpha^n})] .
\end{align}
Let us now estimate $\textbf{E}_t[|\mathcal{M}_T^{\Tilde{\mu}_1, \Tilde{\alpha}^n}-\mathcal{M}_T^{\Tilde{\mu}_2, \alpha^n}|^2].$\\
Since $\mathcal{M}_T^{\Tilde{\mu}_1, \Tilde{\alpha}^n}$ and $\mathcal{M}_T^{\Tilde{\mu}_2, \alpha^n}$ belong to $[0,1]$ and by construction $\mathcal{M}_T^{\Tilde{\mu}_1, \Tilde{\alpha}^n} \geq \mathcal{M}_T^{\Tilde{\mu}_2, \alpha^n}$ a.s., we obtain:
\begin{equation}\label{eq1}
\textbf{E}_t[|\mathcal{M}_T^{\Tilde{\mu}_1, \Tilde{\alpha}^n}-\mathcal{M}_T^{\Tilde{\mu}_2, \alpha^n}|^2] \leq \textbf{E}_t[\mathcal{M}_T^{\Tilde{\mu}_1, \Tilde{\alpha}^n}-\mathcal{M}_T^{\Tilde{\mu}_2, \alpha^n}].
\end{equation}
A similar linearization technique as in the  proof of the Comparison Theorem for BSDEs (see for e.g. \cite{15}) yields:
\begin{align}\label{eqs2}
\Tilde{\mu}_1-\Tilde{\mu}_2 \geq \textbf{E}_t \left[ H_{t,T}^n(\mathcal{M}_T^{\Tilde{\mu}_1, \Tilde{\alpha}^n}-\mathcal{M}_T^{\Tilde{\mu}_2, \alpha^n})| \mathcal{F}_t \right]\, \text{ a.s.},
\end{align}
where $(H^n_{t,s})_{s \in [t,T]}$ is the square integrable process satisfying
\begin{align*}
dH^n_{t,s}=H_{t,s}^n\left[\delta_s^n ds+ \beta_s^n dW_s \right]; \,\,\, H^n_{t,t}=1,
\end{align*}
with
\begin{equation*}
\begin{cases}
\delta_t^n:=\displaystyle\frac{f(t, \mathcal{M}_t^{\Tilde{\mu}_1, \Tilde{\alpha}^n},\Tilde{\alpha}_t^n )-f(t, \mathcal{M}_t^{\Tilde{\mu}_2, \alpha^n},\Tilde{\alpha}_t^n )}{\mathcal{M}_t^{\Tilde{\mu}_1, \Tilde{\alpha}^n}-\mathcal{M}_t^{\Tilde{\mu}_2, \alpha^n}}\textbf{1}_{\{ \mathcal{M}_t^{\Tilde{\mu}_1, \Tilde{\alpha}^n} \neq \mathcal{M}_t^{\Tilde{\mu}_2, \alpha^n}\}}; \\\displaystyle
\beta_t^n:=\frac{f(t, \mathcal{M}_t^{\Tilde{\mu}_2, \alpha^n},\Tilde{\alpha}_t^n )-f(t, \mathcal{M}_t^{\Tilde{\mu}_2, \alpha^n},\alpha_t^n )}{|\Tilde{\alpha}_t^n-\alpha_t^n|^2}(\Tilde{\alpha}_t^n-\alpha_t^n)\textbf{1}_{\Tilde{\alpha}_t^n \neq \alpha_t^n}.
\end{cases}
\end{equation*}

 Now, from $\eqref{eq1}$ and  the H\"{o}lder  inequality, we obtain:
\begin{align}
\textbf{E}_t[\mathcal{M}_T^{\Tilde{\mu}_1, \Tilde{\alpha}^n}-\mathcal{M}_T^{\Tilde{\mu}_2, \alpha^n}]&=\textbf{E}_t[(H_{t,T}^{n})^{-\frac{1}{2}}(H_{t,T}^{n})^{\frac{1}{2}}( \mathcal{M}_T^{\Tilde{\mu}_1, \Tilde{\alpha}^n}-\mathcal{M}_T^{\Tilde{\mu}_2, \alpha^n})] \nonumber\\&\leq 
\textbf{E}_t\left[(H_{t,T}^{n})^{-1}\right]^{\frac{1}{2}} \textbf{E}_t\left[H_{t,T}^n(\mathcal{M}_T^{\Tilde{\mu}_1, \Tilde{\alpha}^n}-\mathcal{M}_T^{\Tilde{\mu}_2, \alpha^n})^2 \right]^{\frac{1}{2}}.
\end{align}
Note that $(\delta^n)_n, (\beta^n)_n$ are predictable process bounded by $C_f$, the Lipschitz constant of $f$. We thus have for all $n \in \mathbb{N}$, $\textbf{E}_t\left[(H_{t,T}^{n})^{-1}\right]\leq C$, for some $C>0$ depending on $C_f$ and $T$ (by the properties of exponential martingales).\\
The above relation together with \eqref{eqs2} and the fact that $\mathcal{M}_T^{\Tilde{\mu}_1, \Tilde{\alpha}^n}-\mathcal{M}_T^{\Tilde{\mu}_2, \alpha^n}$ takes values in $[0,1]$ a.s., imply:
\begin{align}\label{eqp1}
\textbf{E}_t[|\mathcal{M}_T^{\Tilde{\mu}_1,  \Tilde{\alpha}^n}-\mathcal{M}_T^{\Tilde{\mu}_2, \alpha^n}|^2] \leq C (\Tilde{\mu}_1-\Tilde{\mu}_2)^{\frac{1}{2}},
\end{align}
where $C$ is a constant depending on the Lipschitz constant of the driver $f$.\\
By letting $n$ tend to infinity in inequality \eqref{eqp} and using \eqref{eqp1},  we get:
\begin{equation}\label{concl}
|\mathcal{Y}_t(\Tilde{\mu}_1)-\mathcal{Y}_t(\Tilde{\mu}_2)| \leq Err_t \left( \Delta(\Tilde{\mu}_1,\Tilde{\mu}_2) \right).
\end{equation}
Same arguments as in Step 2 of the proof of Theorem \ref{monotonicity} lead to:
\begin{equation}
|\mathcal{Y}_t(\mu_1)-\mathcal{Y}_t(\mu_2)| \leq Err_t \left( \Delta(\mu_1,\mu_2) \right).
\end{equation}

%
%
%
\end{proof}

\subsection{Convexity}

In this section, we provide a convexity result adapted to the non-markovian setting which is established for the map $\mu \rightarrow \mathcal{Y}_t(\mu)$, for any $t<T$.
We extend the results of \cite{BER} to the case of nonlinear constraints, which lead to nontrivial additional technicalities. An important difficulty in our context is represented by the fact that the admissibility set is given by the two processes $\mathcal{E}^f[0]$ and $\mathcal{E}^f[1]$, contrary to $\cite{BER}$ where it is given by the two constants $0$ and $1$.\\

We first recall the notion of $\mathcal{F}_t$ - convexity introduced in \cite{BER}.

\begin{definition}[$\mathcal{F}_t$-convexity]

\begin{itemize}
\item[$(i)$] We say that a subset $D \subset \textbf{L}_{2}(\mathbf{R}, \mathcal{F}_t)$ is $\mathcal{F}_t$-convex if for all $\mu_1, \mu_2 \in D$ and $\lambda \in \textbf{L}_0([0,1], \mathcal{F}_t)$, $\lambda \mu_1+(1-\lambda)\mu_2 \in D$.
\item[$(ii)$] Let $D$ be an $\mathcal{F}_t$-convex subset of $\textbf{L}_{2}(\mathbf{R}, \mathcal{F}_t)$. A map $\mathcal{J}:D \mapsto \textbf{L}_2(\mathbf{R}, \mathcal{F}_t)$ is said to be $\mathcal{F}_t$-convex if
$$Epi(\mathcal{J}):=\{(\mu,Y) \in D \times \textbf{L}_2(\mathbf{R}, \mathcal{F}_t): Y \geq \mathcal{J}(\mu)   \}$$
is $\mathcal{F}_t$-convex.
\item[$(iii)$] Let $Epi^c(\mathcal{J})$ be the set of elements of the form $\sum_{n \leq N} \lambda_n(\mu_n, Y_n)$ with $(\mu_n, Y_n, \lambda_n)_{n \leq N} \subset Epi(\mathcal{J}) \times \textbf{L}_0([0,1], \mathcal{F}_t)$ such that $\sum_{n \leq N} \lambda_n =1$, for some $N \geq 1.$ We then denote by $\overline{Epi}^c(\mathcal{J})$ its closure in $\textbf{L}_2$. The $\mathcal{F}_t$-convex envelope of $\mathcal{J}_t$ is defined as
\begin{align}
\mathcal{J}_t^c(\mu):=ess\inf \{Y \in \textbf{L}_2(\mathbf{R}, \mathcal{F}_t): (\mu,Y) \in  \overline{Epi}^c(\mathcal{J}_t)\}.
\end{align}
\end{itemize}
\end{definition}

\begin{assumption}\label{assumpt1}
We assume that the map $\Phi$ is Lipschitz continuous in $x$, uniformly with respect to $\omega$.
\end{assumption}

\begin{proposition}
Under Assumption \ref{assumpt1}, the map $\mu \in \textbf{D}_t  \mapsto \mathcal{Y}_{t}(\mu)$ is $\mathcal{F}_t$-convex, for all $t<T$.
\end{proposition}
The proof is divided in several steps. We follow the arguments used in the proof of Proposition 3.2 in \cite{BER} up to non trivial modifications due to the  nonlinearity of the driver $f$. The technical arguments specific to the nonlinear case are mostly needed in Step 5 of the proof. For  convienence of the reader, we also present the main ideas of Steps 1-4.\\


\begin{proof}

1.
 $(\mu, \mathcal{Y}_t^c(\mu)) \in \overline{Epi}^c(\mathcal{Y}_t)$, for all $\mu \in \textbf{D}_t.$\\
For every fixed element $\mu \in \textbf{D}_t$, the family $F:=\{Y \in \textbf{L}_2(\mathbf{R}, \mathcal{F}_t): (\mu,Y) \in \overline{Epi}^c(\mathcal{Y}_t) \}$ is direct downward since $Y^1 \textbf{1}_{\{  Y^1 \leq Y^2\}}+Y^2 \textbf{1}_{\{Y^1>Y^2\}} \in F$ for all $Y^1$, $Y^2$, by $\mathcal{F}_t$-convexity of $\overline{Epi}^c(\mathcal{Y}_t)$. It then follows that we can find a sequence $(Y^n)_{n \geq 1} \subset F$ such that $Y^n \downarrow \mathcal{Y}_t^c(\mu)$ a.s. Moreover, $Y^1$ and $\mathcal{Y}_t^c(\mu)$ belong to $\textbf{L}_2$, and thus the monotone convergence Theorem leads to $Y^n \rightarrow \mathcal{Y}_t^c(\mu)$ in $\textbf{L}_2$, as $n$ goes to infinity. The set $\overline{Epi}^c(\mathcal{Y}_t)$ is closed in $\textbf{L}_2$ and hence the result follows.\\

2. Let $\eta \in \mathbf{S}_2$ be as in  $\eqref{B}$. Then, $|\mathcal{Y}_t^c(\mu)| \leq \eta_t$, for all $t \leq T$ and $\mu \in \textbf{D}_t.$\\
We first show that $\mathcal{Y}_t^c(\mu) \geq - \eta_t.$ By Point 1, it follows that $(\mu, \mathcal{Y}_t^c(\mu)) \in \overline{Epi}^c(\mathcal{Y}_t)$  is obtaind as $\textbf{L}_2$-limit of elements of the form $\sum_{n \leq N} \lambda_n(\mu_n,Y_n)$ with $(\mu_n, Y_n, \lambda_n) \subset Epi(\mathcal{Y}_t) \times \textbf{L}_0([0,1], \mathcal{F}_t),$ such that $\sum_{n \leq N} \lambda_n=1.$ Inequality $\ref{B}$  implies that each $Y^n$ of the above family is bounded below by $-\eta_t$ and hence this also holds for $\mathcal{Y}_t^c(\mu)$. The converse inequality $\mathcal{Y}_t^c \leq \eta_t$ is clear since \eqref{B} holds and, by construction, $\mathcal{Y} \geq \mathcal{Y}_t^c$.\\

3. The map $\mu \in \textbf{D}_t \mapsto \mathcal{Y}_t^c(\mu)$ is $\mathcal{F}_t$-convex.\\
We have the show that $Epi(\mathcal{Y}_t^c)$ is $\mathcal{F}_t-$convex. Let us fix $\mu^{1}, \mu^{2} \in \textbf{D}_t$ and $\lambda \in \textbf{L}_0([0,1], \mathcal{F}_t)$. Since $\overline{Epi}^c(\mathcal{Y}_t)$ is $\mathcal{F}_t$-convex and $(\mu^{i}, \mathcal{Y}_t^c(\mu^{i})) \in \overline{Epi}^c(\mathcal{Y}_t)$, for $i=1,2$, it follows that $(\lambda \mu^{1}+(1-\lambda)  \mu^{2}, \lambda \mathcal{Y}_t^c(\mu^1)+ (1-\lambda)\mathcal{Y}_t^c(\mu^2)) \in \overline{Epi}^c(\mathcal{Y}_t)$, and thus $\lambda \mathcal{Y}_t^c(\mu^1)+(1-\lambda)\mathcal{Y}_t^c(\mu^2) \geq \mathcal{Y}_t^c(\lambda \mu^1+(1-\lambda) \mu^2,$ by definition of $\mathcal{Y}_t^c(\mu)$. We obtain that $\lambda Y^1+(1-\lambda)Y^2 \geq \mathcal{Y}_t^c(\lambda \mu^1+ (1-\lambda)\mu^2)$, for any $Y^1, Y^2$ such that $(\mu^i,Y^{i}) \in Epi(\mathcal{Y}_t^c)$, $i=1,2$. The result follows.\\
\bigskip

4. $\mathcal{Y}_{t}(\mu) \geq \mathcal{Y}_t^c(\mu)$, for all $\mu \in \textbf{D}_t.$\\
Let $(\mu_n)_n \in \textbf{D}_t$ be such that $\mu_n \rightarrow \mu$ a.s. when $n \rightarrow \infty.$ Recall that under Assumption \ref{assumpt1}, the map $\mu \rightarrow \mathcal{Y}_t(\mu)$ is a.s. continuous and hence $\mathcal{Y}_t(\mu_n) \rightarrow \mathcal{Y}_t(\mu)$ a.s. when $n \rightarrow \infty$. Moreover, by \ref{B} we have $\mathcal{Y}_t(\mu_n) \rightarrow \mathcal{Y}_t(\mu)$ in $\mathbf{L}_2$. Note that $Epi(\mathcal{Y}_t) \subset \overline{Epi}^c(\mathcal{Y}_t)$ and thus $(\mu, \mathcal{Y}_{t}(\mu)) \in \overline{Epi}^c(\mathcal{Y}_t).$ The result follows by using the definition of $\mathcal{Y}_t^c$.\\


\bigskip
5. $\mathcal{Y}_t^c(\mu) \geq \mathcal{Y}_{t}(\mu)$, for all $\mu \in \textbf{D}_t.$\\

(i) It follows from Point 1, that there exists a sequence
$$(\mu_n, Y_n, \lambda_n^N)_{n \geq 1, N \geq 1} \subset Epi(\mathcal{Y}_t) \times \textbf{L}_0([0,1], \mathcal{F}_t)$$
such that $\sum_{n \leq N} \lambda_n^N=1$, for all $N$, and
\begin{align}\label{conv}
(\hat{\mu}_N, \hat{Y}_N):=\sum_{n \leq N} \lambda_n^N(\mu_n,Y_n) \mapsto (\mu, \mathcal{Y}_t^c) \in \textbf{L}_2.
\end{align}

Fix $N \geq 1$ and $M \geq 1$. We claim that $\mathcal{Y}_t(\hat{\mu}_N) \leq \hat{Y}_N.$ The proof is postponed to Step 5, point (ii).
We deduce:
\begin{align*}
\lim \inf_{N \rightarrow \infty} \mathcal{Y}_t(\hat{\mu}_N) \leq \mathcal{Y}_t^c(\mu).
\end{align*}
We now define:
\begin{align*}
Z_M(\mu):= \essinf \{\mathcal{Y}_t(\mu'): \,\,\, |\mu'-\mu| \leq \frac{1}{M}\}.
\end{align*}
and set $D_{\mu}^{M}:= \{\mu' \in \mathbf{D}_t: \,\,|\mu'-\mu|\leq \frac{1}{M}\}.$
By Lemma \ref{sequence}, it exists a sequence $(\mu_n^M)_n$ with $\mu_n^M \in D_\mu^M$ for all $n$ such that
\begin{align}\label{eqconv}
\mathcal{Y}_t(\mu_n^M) \rightarrow Z_M(\mu) \text{ a.s. when } n \rightarrow \infty.
\end{align}
  
One can easily remark that under Assumption \ref{assumpt1}, the estimate given in Theorem \ref{cont}  becomes:
\begin{align}\label{eqconv1}
|\mathcal{Y}_t(\mu_n^M)-\mathcal{Y}_t(\mu)| \leq Err_t(\Delta |\mu_n^M-\mu|) \leq K |\mu_n^M-\mu|^{\frac{1}{4}} \leq K \frac{1}{M^\frac{1}{4}},
\end{align}
where $K$ is a constant depending on $C_f, T$ and the Lipschitz constant of $\Phi$.\\
Note that:
\begin{align*}
|\mathcal{Y}_t(\mu)-Z_M(\mu)|\leq |\mathcal{Y}_t(\mu)-\mathcal{Y}_t(\mu_n^M)|+|\mathcal{Y}_t(\mu_n^M)-Z_M(\mu)|.
\end{align*}

Coupling the above inequality with $\eqref{eqconv1}$and $\eqref{eqconv}$, letting first $n$ and then $M$ to $\infty$, we get 
\begin{align}\label{convconv}
Z_M(\mu) \rightarrow \mathcal{Y}_t(\mu) \text{ a.s. when } M \rightarrow +\infty.
\end{align}

Now, the convergence  $\hat{\mu}_N \rightarrow \mu$ a.s. and Lemma \ref{R} imply that:
\begin{align}
Z_M(\mu) \leq \lim \inf_{N \rightarrow \infty} \mathcal{Y}_t(\bar{\mu}_N)= \lim \inf_{N \rightarrow \infty} \left(\mathcal{Y}_t(\hat{\mu}_N) \textbf{1}_{|\hat{\mu}_N-\mu| \leq \frac{1}{M}}+\mathcal{Y}_t(\mu) \textbf{1}_{|\hat{\mu}_N-\mu|>\frac{1}{M}}\right) \leq \mathcal{Y}_t^c(\mu),
\end{align}
where:
$$\bar{\mu}_N:= \hat{\mu}_N \textbf{1}_{|\hat{\mu}_N-\mu| \leq\frac{1}{M}}+\mu \textbf{1}_{|\hat{\mu}_N-\mu|>\frac{1}{M}} \in D_{\mu}^{M}.$$
Also, since by \eqref{convconv}, $Z_M(\mu) \uparrow \mathcal{Y}_{t}(\mu)$ as $M$ goes to $+\infty$, the result follows.\\

(ii)
It remains to prove:
\begin{align}\label{estim}
\mathcal{Y}_t(\hat{\mu}_N) \leq \hat{Y}_N.
\end{align}

Fix $\varepsilon >0$. Let us consider a random variable, $\mathcal{F}_{t+\varepsilon}$ measurable $\zeta_N^{\varepsilon}$ such that $P\left[\zeta_N^{\varepsilon}= \mathcal{M}_{t+\varepsilon}^{\mu_n, \alpha_n}| \mathcal{F}_t\right]=\lambda_n^N$, where $\alpha_n \in \textbf{A}_{t, \mu_n}$. Clearly, by construction, $\zeta_N^{\varepsilon}$ belongs to $\left[\mathcal{E}_{t+\varepsilon,T}^f[0], \mathcal{E}_{t+\varepsilon,T}^f[1]\right]$ a.s.
We set:
\begin{align}\label{defmu}
\mu_{N}^{\varepsilon}:= \mathcal{E}_{t,t+\varepsilon}^f \left[\zeta_N^\varepsilon \right].
\end{align}
We rewrite $\mathcal{Y}_t(\hat{\mu}_N)$ as follows:
\begin{align}
\mathcal{Y}_t(\hat{\mu}_N)= \mathcal{Y}_t(\hat{\mu}_N)-\mathcal{Y}_t({\mu}_N^\varepsilon)+\mathcal{Y}_t({\mu}_N^\varepsilon)
\end{align}
and by appealing to  Theorem $\ref{cont}$, we obtain:
\begin{align}\label{em1}
\mathcal{Y}_t(\hat{\mu}_N) \leq Err_t(\Delta(\hat{\mu}_N-\mu_N^\varepsilon))+\mathcal{Y}_t({\mu}_N^\varepsilon).
\end{align}

We now show that $\lim \sup_{\varepsilon \rightarrow 0} \left[Err_t(\Delta(\hat{\mu}_N-\mu_N^\varepsilon))+\mathcal{Y}_t({\mu}_N^\varepsilon)\right]\leq \hat{Y}_N.$

To this purpose, we split the proof in several steps:

\bigskip

\textbf{Step a.}  We prove that $\lim_{\varepsilon \rightarrow 0}Err_t(\Delta(\hat{\mu}_N-\mu_N^\varepsilon))=0$ a.s.\\

We start by showing that
$\lim_{\varepsilon \rightarrow 0} \mu_N^{\varepsilon}= \hat{\mu}_N$ a.s.

Since $(\mu_n)_{n \leq N}$ are $\mathcal{F}_t$-measurable and $P\left[\zeta_N^{\varepsilon}=
\mathcal{M}_{t+\varepsilon}^{\mu_n, \alpha_n}|\mathcal{F}_t\right]=\lambda_n^N$, we have  $$\hat{\mu}_N= \textbf{E}_t[\sum_{n \leq N} \textbf{1}_{\zeta_N^{\varepsilon}=
\mathcal{M}_{t+\varepsilon}^{\mu_n, \alpha_n}} \mu_n]$$ a.s.
We split the difference between $\mu_{N}^{\varepsilon}$ and $\mu_N$ in two terms as follows:
\begin{align}\label{l1}
|\mu_N^\varepsilon-\hat{\mu}_N|&=|\mathcal{E}_{t,t+\varepsilon}^{f}[\sum_{n \leq N} \textbf{1}_{\zeta_N^{\varepsilon}=
\mathcal{M}_{t+\varepsilon}^{\mu_n, \alpha_n}} \mathcal{M}_{t+\varepsilon}^{\mu_n, \alpha_n}]-\textbf{E}_{t}[\sum_{n \leq N} \textbf{1}_{\zeta_N^{\varepsilon}=
\mathcal{M}_{t+\varepsilon}^{\mu_n, \alpha_n}}\mu_n]|^2 \nonumber\\&\leq
 2|\mathcal{E}_{t,t+\varepsilon}^{f}[\sum_{n \leq N} \textbf{1}_{\zeta_N^{\varepsilon}=
\mathcal{M}_{t+\varepsilon}^{\mu_n, \alpha_n}}\mathcal{M}_{t+\varepsilon}^{\mu_n, \alpha_n}]-\mathcal{E}_{t,t+\varepsilon}^{f}[\sum_{n \leq N} \textbf{1}_{\zeta_N^{\varepsilon}=
\mathcal{M}_{t+\varepsilon}^{\mu_n, \alpha_n}}\mu_n]|^2\\&+2|\mathcal{E}_{t,t+\varepsilon}^{f}[\sum_{n \leq N}\textbf{1}_{\zeta_N^{\varepsilon}=
\mathcal{M}_{t+\varepsilon}^{\mu_n, \alpha_n}}\mu_n]-\textbf{E}_{t}[\sum_{n \leq N} \textbf{1}_{\zeta_N^{\varepsilon}=
\mathcal{M}_{t+\varepsilon}^{\mu_n, \alpha_n}}\mu_n]|^2. \nonumber\end{align}
From the a priori estimations on BSDEs, we obtain:
\begin{align}\label{leq2}
&|\mathcal{E}_{t,t+\varepsilon}^{f}[\sum_{n \leq N} \textbf{1}_{\zeta_N^{\varepsilon}=
\mathcal{M}_{t+\varepsilon}^{\mu_n, \alpha_n}} \mathcal{M}_{t+\varepsilon}^{\mu_n, \alpha_n}]-\mathcal{E}_{t,t+\varepsilon}^{f}[\sum_{n \leq N} \textbf{1}_{\zeta_N^{\varepsilon}=
\mathcal{M}_{t+\varepsilon}^{\mu_n, \alpha_n}}\mu_n]|^2 \nonumber\\&\qquad \leq \textbf{E}_t[\sum_{n \leq N} \textbf{1}_{\zeta_N^{\varepsilon}=
\mathcal{M}_{t+\varepsilon}^{\mu_n, \alpha_n}}(\mathcal{M}_{t+\varepsilon}^{\mu_n, \alpha_n}-\mu_n)^2] \leq
\sum_{n \leq N} \textbf{E}_t[ (\mathcal{M}_{t+\varepsilon}^{\mu_n, \alpha_n}-\mu_n)^2].
\end{align}

Since for all $n \leq N$ the processes $\mathcal{M}_\cdot^{\mu_n, \alpha_n}$ are continuous and belong to $\mathbf{S}_2$, we can apply  Lebesgue's  theorem and obtain that the right member of \eqref{leq2} tends to $0$ when $\varepsilon \rightarrow 0.$
Moreover, by applying Proposition \ref{estimation} with $\xi^\varepsilon=\sum_{n \geq 1} \textbf{1}_{A_n^\varepsilon}\mu_n$, we derive that it exists $\eta_\varepsilon$, with $\eta_\varepsilon \rightarrow 0$ a.s. when $\varepsilon \rightarrow 0$ such that:
\begin{align}\label{l3}
|\mathcal{E}_{t,t+\varepsilon}^{f}[\sum_{n \leq N} \textbf{1}_{\zeta_N^{\varepsilon}=
\mathcal{M}_{t+\varepsilon}^{\mu_n, \alpha_n}}\mu_n]-\textbf{E}_{t}[\sum_{n \leq N} \textbf{1}_{\zeta_N^{\varepsilon}=
\mathcal{M}_{t+\varepsilon}^{\mu_n, \alpha_n}}\mu_n]|^2 \leq \eta_\varepsilon.
\end{align}
From \eqref{l1}, \eqref{leq2} and \eqref{l3}, by letting $\varepsilon$ tend to $0$, we get that $\lim_{\varepsilon \rightarrow 0} \mu_N^{\varepsilon}= \hat{\mu}_N$ a.s. This implies that $\Delta(\hat{\mu}_N-\mu_N^\varepsilon) =C|\hat{\mu}_N-\mu_N^\varepsilon|^{\frac{1}{2}}\rightarrow 0.$ Since $\Phi$  satisfies Assumption \ref{assumpt1},  we get by  Theorem \ref{cont}, the desired result.\\

\textbf{Step b.} We prove that for each $n \leq N$, $\lim_{\varepsilon \rightarrow 0}\textbf{E}_t[|\mathcal{Y}_{t+\varepsilon}(\mathcal{M}_{t+\varepsilon}^{\mu_n, \alpha_n})-\mathcal{Y}_t(\mu_n)|]=0$ a.s.\\

As Assumption $\ref{assumpt1}$, inequality \ref{B} and Remark $\ref{RRR}$ hold, we can apply Theorem \ref{cadlag} and Lebesgue's  Theorem, which lead to the desired result.\\

\textbf{Step c.}

Recall that by $\eqref{defmu}$ we have $\mu_{N}^{\varepsilon}= \mathcal{E}_{t,t+\varepsilon}^f \left[\zeta_N^\varepsilon \right].$ Lemma \ref{DPP} gives:
\begin{align*}
\mathcal{Y}_t({\mu}_N^\varepsilon) \leq \mathcal{E}_{t,t+\varepsilon}^g[\mathcal{Y}_{t+\varepsilon}(\zeta_N^\varepsilon)] =  \mathcal{E}_{t,t+\varepsilon}^g \left( \mathcal{Y}_{t+\varepsilon}(\sum_{n \leq N} \textbf{1}_{\zeta_N^{\varepsilon}=
\mathcal{M}_{t+\varepsilon}^{\mu_n, \alpha_n}}\mathcal{M}_{t+\varepsilon}^{\mu_n, \alpha_n}) \right).
\end{align*}

By Lemma \ref{R}, we obtain:
\begin{align}\label{em2}
\mathcal{Y}_t({\mu}_N^\varepsilon) \leq \mathcal{E}_{t,t+\varepsilon}^g \left(\sum_{n \leq N} \textbf{1}_{\zeta_N^{\varepsilon}=
\mathcal{M}_{t+\varepsilon}^{\mu_n, \alpha_n}} \mathcal{Y}_{t+\varepsilon}(\mathcal{M}_{t+\varepsilon}^{\mu_n, \alpha_n}) \right).
\end{align}
We now apply Proposition \ref{estimation} with $\xi^\varepsilon:=\sum_{n \geq 1} \textbf{1}_{\zeta_N^{\varepsilon}=
\mathcal{M}_{t+\varepsilon}^{\mu_n, \alpha_n}}\mathcal{Y}_{t+\varepsilon}(\mathcal{M}_{t+\varepsilon}^{\mu_n, \alpha_n})$ and  derive that it exists $\eta'_\varepsilon$, with $\eta'_\varepsilon \rightarrow 0$ a.s. when $\varepsilon \rightarrow 0$ such that:
$$\mathcal{Y}_t({\mu}_N^\varepsilon) \leq \mathcal{E}_{t,t+\varepsilon}^g \left(\sum_{n \leq N} \textbf{1}_{\zeta_N^{\varepsilon}=
\mathcal{M}_{t+\varepsilon}^{\mu_n, \alpha_n}} \mathcal{Y}_{t+\varepsilon}(\mathcal{M}_{t+\varepsilon}^{\mu_n, \alpha_n}) \right)\leq   \textbf{E}_{t}\left[\sum_{n \leq N} \textbf{1}_{\zeta_N^{\varepsilon}=
\mathcal{M}_{t+\varepsilon}^{\mu_n, \alpha_n}} \mathcal{Y}_{t+\varepsilon}(\mathcal{M}_{t+\varepsilon}^{\mu_n, \alpha_n}) \right]+\eta'_\varepsilon.$$
We finally get:
$$\mathcal{Y}_t({\mu}_N^\varepsilon)\leq \textbf{E}_t \left[\sum_{n \leq N}\textbf{1}_{\zeta_N^{\varepsilon}=
\mathcal{M}_{t+\varepsilon}^{\mu_n, \alpha_n}} \mathcal{Y}_t(\mu_n)\right]+\sum_{n \leq N} \textbf{E}_t \left[|\mathcal{Y}_{t+\varepsilon}(\mathcal{M}_{t+\varepsilon}^{\mu_n, \alpha_n})-\mathcal{Y}_t(\mu_n)|\right]+\eta'_\varepsilon. $$
Letting $\varepsilon$ tend to $0$ in the above inequality, we obtain, by Step b:
\begin{align}\label{last}
\lim \sup_{\varepsilon \rightarrow 0} \mathcal{Y}_t({\mu}_N^\varepsilon) \leq  \sum_{n \leq N} \lambda_n^N \mathcal{Y}_t(\mu_n) \leq \sum_{n \leq N} \lambda_n^N Y_n= \hat{Y}_n,
\end{align}
where the last inequality follows by definition of the sequence $(\mu_n, Y_n)_{n}$ and $\eqref{conv}.$

The desired result \eqref{estim} is obtained by combining \eqref{em1}  with Step a and \eqref{last}.

\end{proof}

\begin{remark}
We recall that in \cite{BER} the authors do not assume the continuity of the map $\mu \rightarrow \mathcal{Y}_{t}(\mu)$ and obtain the convexity of the lower semi-continuous envelope  $\mathcal{Y}_{t^*}(\mu)$, which is defined:
$$\mathcal{Y}_{t^*}(\mu):= \lim_{\varepsilon \rightarrow 0} \essinf\{\mathcal{Y}_t(\mu'): \,\, |\mu'-\mu|\leq \varepsilon\}. $$
In our nonlinear setting, using exactly the same arguments as above, the fact that $\lim \inf_{N \rightarrow \infty} \mathcal{Y}_{t^*}(\mu_N) \geq \mathcal{Y}_{t^*}(\mu)$ when $\mu_N \rightarrow \mu$ a.s. and $\mathcal{Y}_{t^*}(\mu) \leq \mathcal{Y}_{t}(\mu)$, we obtain the convexity of the lower-semicontinuous envelope  $\mathcal{Y}_{t^*}(\mu)$ as in $\cite{BER}$.
\end{remark}

\section{Dual representation in the case of concave constraints}

We now provide a dual representation of the value function defined by \eqref{probl}, which takes the form of a stochastic control problem in Meyer form. The results of this section extend the ones given in \cite{BER}, but involve technical additional proofs, due to  the nonlinearity of the coefficient $f$.

For each $(\omega,t)$, let $\Tilde{f}(\omega,t, \cdot, \cdot, \cdot)$  be
be the concave conjugate  of $f$  with respect to $(x,\pi)$, defined for each $(p,q)$ in $\mathbf{R} \times \mathbf{R}^d$ as follows:

$$\Tilde{f}: (\omega,t,p,q) \in \Omega \times [0,T] \times \mathbf{R} \times \mathbf{R}^d  \rightarrow  \inf_{(x,\pi) \in \mathbf{R} \times \mathbf{R}^d }\left(xp+\pi ^\top q-f(\omega,t,x,\pi)\right). $$

For each $(\omega,t)$, we denote by $\Tilde{g}(\omega,t, \cdot, \cdot, \cdot)$ the convexe conjugate  of $g$  with respect to $(y,z)$, defined for each $(u,v)$ in $\mathbf{R} \times \mathbf{R}^d$ as follows:

$$\Tilde{g}: (\omega,t,u,v) \in \Omega \times [0,T] \times \mathbf{R} \times \mathbf{R}^d  \rightarrow  \sup_{(y,z) \in \mathbf{R} \times \mathbf{R}^d} \left(yu+z^\top v-g(\omega,t,y,z)\right).$$
We also introduce for each $\omega$, the polar function of $\Phi$ with respect to $m$:
$$\Tilde{\Phi}:(\omega,l) \in \Omega \times \mathbf{R} \rightarrow \sup_{m \in [0,1]} \left(ml-\Phi(\omega,m)\right).$$

In the sequel, we denote by  $\mathcal{U}$ the set of predictable processes valued in $D^1$, respectively by $\mathcal{V}$ the set  of predictable processes valued in $D_t^2$, where for each $(t, \omega) \in [0,T] \times \Omega$, $D^1_t(\omega)$ and  $D_t^2(\omega)$ are defined as follows:
\begin{align}
D_t^1(\omega):= \{(p,q): \,\, \Tilde{f}(t,\omega,p,q) > -\infty  \}; \,\,\, D_t^2(\omega):= \{(u,v): \,\, \Tilde{g}(t,\omega,u,v) < +\infty\}.
\end{align}

\begin{remark}\label{boundness}
For each $(t, \omega),$ $ D_t^1(\omega) \subset U$, where $U$ is the closed subset of  $\mathbf{R} \times \mathbf{R}^d $ of elements $\alpha=(\alpha_1, \alpha_2)$ such that $|\alpha_1| \leq C_g$ and $|\alpha_2^{i}|\leq C_g$, $\forall i=\overline{1,d}.$
The same remark holds for the elements belonging to $D_t^2(\omega)$, with $C_f$ instead of $C_g$.
\end{remark}
To each $l>0$, $\gamma=(\kappa, \vartheta) \in \mathcal{V}$  (resp. $\lambda=(\mu, \nu) \in \mathcal{U}$ ), we associate the processes $\mathcal{A}^{l,\gamma}$ (resp.   $\mathcal{L}^{\lambda}$ ) defined by
$$\mathcal{A}_t^{l, \gamma}=l+\int_0^t \mathcal{A}_s^{l, \gamma} \kappa_s ds+\int_0^t \mathcal{A}_s^{l,\gamma} \vartheta_s dW_s, \text{  } t \in [0,T];$$
$$\mathcal{L}_t^{\lambda}=1+\int_0^t \mathcal{L}_s^{\lambda} \mu_s ds+\int_0^t \mathcal{L}_{s}^{\lambda} \nu_s dW_s, \text{  } t \in [0,T].$$

The dual formulation of $\mathcal{Y}_0$ is expressed in terms of
$$\mathcal{X}_0(l):=\inf_{(\lambda, \gamma) \in \mathcal{U} \times \mathcal{V}} X_0^{l,\lambda, \gamma} $$
where
$$X_0^{l, \lambda, \gamma}:=E \left[\int_0^T \mathcal{L}_s^{\lambda}\Tilde{g}(s,\lambda_s)ds-\int_0^T\mathcal{A}_s^{l, \gamma} \Tilde{f}(s,\gamma_s)ds+\mathcal{L}_T^{\lambda}\Tilde{\Phi}(\dfrac{\mathcal{A}_T^{l, \gamma}}{\mathcal{L}_T^{\lambda}}) \right].$$

\begin{proposition}\label{f11}
$\mathcal{Y}_0(m) \geq \sup_{l>0}(lm-\mathcal{X}_0(l))$, for all $m \in \left[\mathcal{E}_{0,T}^f[0],\mathcal{E}_{0,T}^f[1]\right].$
\begin{proof}
Fix $\alpha \in \textbf{A}_{0,m}$, $\lambda=(\nu,\mu) \in \mathcal{U}$, $l>0$ and $\gamma=(\kappa,\vartheta) \in \mathcal{V}$. The definition of $\Tilde{\Phi}$, together with Ito formula imply:
\begin{equation}
\textbf{E}[Y_T^{m,\alpha} \mathcal{L}_T^{\lambda}] \leq Y_0^{m, \alpha}+\textbf{E}[\int_0^T \mathcal{L}_s^{\lambda}\Tilde{g}(s,\lambda_s)ds]
\end{equation}
and
\begin{align}
\textbf{E}[Y_T^{m,\alpha}\mathcal{L}_T^{\lambda}]&=\textbf{E}[\Phi(\mathcal{M}_T^{m,\alpha})\mathcal{L}_T^{\lambda}] \geq \textbf{E}[\mathcal{A}_T^{l,\gamma} \mathcal{M}_T^{m,\alpha}-\mathcal{L}_T^{\lambda}\Tilde{\Phi}(\dfrac{\mathcal{A}_T^{l,\gamma}} {\mathcal{L}_T^{\lambda}})]\nonumber\\
& \geq \textbf{E}[lm+\int_0^T\mathcal{A}_s^{l,\gamma} \Tilde{f}(s,\gamma_s)ds-\mathcal{L}_T^{\lambda}\Tilde{\Phi}(\dfrac{\mathcal{A}_T^{l,\gamma}} {\mathcal{L}_T^{\lambda}})].
\end{align}
Note that since $Y^{m,\alpha}, \mathcal{L}^\lambda, \mathcal{M}^{m,\alpha}, \mathcal{A}^{l, \gamma} \in \mathbf{S}_2$, $Z^{m,\alpha}, \alpha \in \textbf{H}_2$ and Remark \ref{boundness} holds, by applying Burkholder-Davis-Gundy inequality, we obtain that the local martingales $\int_0^\cdot Y_s^{m,\alpha}\mathcal{L}_s^\lambda \nu_s^\top dW_s,$ $\int_0^\cdot \mathcal{L}_s^\lambda Z_s^{m,\alpha, \top }dW_s$, $\int_0^\cdot \mathcal{M}_s^{m,\alpha}\mathcal{A}_s^{l, \gamma} \vartheta_s^\top dW_s,\int_0^\cdot \mathcal{A}_s^{l,\gamma} \alpha_s^\top dW_s$ are in fact martingales. Hence we can cancel their expectations.
From the two above inequalities, we derive that:
$$Y_0^{m, \alpha} \geq lm-\textbf{E}\left[\int_0^T \mathcal{L}_s^{\lambda}\Tilde{g}(s,\lambda_s)ds-\int_0^T\mathcal{A}_s^{l, \gamma} \Tilde{f}(s,\gamma_s)ds+\mathcal{L}_T^{\lambda}\Tilde{\Phi}(\dfrac{\mathcal{A}_T^{l, \gamma}}{\mathcal{L}_T^{\lambda}})\right].$$
By arbitrariness of $(\lambda, \gamma) \in \mathcal{U} \times \mathcal{V}$, we get:
$$Y_0^{m, \alpha} \geq lm-\mathcal{X}_0(l).$$
We then take the essential infimum on $\alpha \in \textbf{A}_0^m$ and  the supremum on $l>0$. The result follows.
\end{proof}
\end{proposition}

We now show that equality holds under some additional assumptions.
\begin{assumption}\label{HYP}
We make the following assumptions:
\begin{itemize}
\item[$(a)$] For each $(t,\omega) \in \Omega \times [0,T]$, the maps $\Tilde{\Phi}(\omega,\cdot)$, $\Tilde{f}(t,\omega,\cdot)$ and $\Tilde{g}(t,\omega, \cdot)$ are of class $C^1_b$.  Also  $D_t^1(\omega)$and $D_t^2(\omega)$ are closed.

\item[$(b)$] $|\nabla \Tilde{\Phi}(\omega,\cdot)|+|\!|\nabla \Tilde f(\omega,t,\cdot)|\!|_{\mathbf{R} \times \mathbf{R}^d}+|\!|\nabla \Tilde g(\omega,t,\cdot)|\!|_{\mathbf{R} \times \mathbf{R}^d} \leq C_{ \Tilde{\Phi},\Tilde f,\Tilde g},$ for some $C_{ \Tilde{\Phi},\Tilde f,\Tilde g} \in \textbf{L}_2(\mathbf{R});$

\item[$(c)$] $\Phi(\omega,m)= \sup_{l >0} \left(lm-\Tilde{\Phi}(\omega,l)\right)$,  for all $m \in [0,1]$;

\item[$(d)$] $f(\omega,t,x,\pi)= \min_{(p,q) \in D^2_t(\omega)} \left(px+\pi ^\top q-\Tilde{f}(\omega,t,p,q)\right)$, for all $(x,\pi) \in \mathbf{R} \times \mathbf{R}^d$;

\item[$(e)$] $g(\omega,t,y,z)= \max_{(u,v) \in D^1_t(\omega)} \left(yu+z^\top v-\Tilde{g}(\omega,t,u,v)\right)$, for all $(y,z) \in \mathbf{R} \times \mathbf{R}^d$.

\end{itemize}

\end{assumption}

\begin{proposition}
Assume that there exists $\hat{l}>0$, $\hat{\lambda} \in \mathcal{U}$ and $\hat{\gamma} \in \mathcal{V}$ such that
\begin{align}
\sup(lm-\mathcal{X}_0(l))=\hat{l}m-\mathcal{X}_0(\hat{l})=\hat{l}m-X_0^{\hat{l},\hat{\lambda},\hat{\gamma}}.
\end{align}
Then there exists $\hat{\alpha} \in \textbf{H}_2$ such that
\begin{align}
\mathcal{Y}_0(m)=Y_0^{m, \hat{\alpha}}=\hat{l}m-\mathcal{X}_0({\hat{l}}).
\end{align}
Also it satisfies
\begin{equation}\label{cond1}
\begin{cases}
f(\cdot, \mathcal{M}^{m, \hat{\alpha}}, \hat{\alpha})=\hat{\kappa}\mathcal{M}^{m, \hat{\alpha}}+\hat{\vartheta}^\top \hat{\alpha}-\Tilde{f}(\cdot,\hat{\gamma}); \ \; \text{  }   \mathcal{M}_T^{m, \hat{\alpha}}=\nabla \Tilde{\Phi}\left(\frac{\mathcal{A}_T^{ \hat{l},\hat{\gamma}}}{\mathcal{L}_T^{\hat{\lambda}}}\right);\\

g(\cdot, Y^{m, \hat{\alpha}},Z^{m, \hat{\alpha}})=\hat{\mu}Y^{m, \hat{\alpha}}+\hat{\nu}^\top Z^{m, \hat{\alpha}}-\Tilde{g}(\cdot, \hat{\lambda}); \ \; \text{ }

\Phi( \mathcal{M}_T^{m, \hat{\alpha}})= \frac{\mathcal{M}_T^{m, \hat{\alpha}} \mathcal{A}_T^{\hat{l},\hat{\gamma}}}{\mathcal{L}_T^{\hat{\lambda}}}-\Tilde{\Phi}(\frac{ \mathcal{A}_T^{\hat{l},\hat{\gamma}}}{\mathcal{L}_T^{\hat{\lambda}}}).
\end{cases}
\end{equation}
\end{proposition}

\begin{proof}

The proof is divided in two steps.\\

\vspace{3mm}
\noindent{\bf Step 1.} We denote by $\left( \mathcal{E}_{\cdot,T}^{f}\left[\nabla \Tilde{\Phi}(\frac{\mathcal{A}_T^{\hat{l}, \hat{\gamma}}}{\mathcal{L}_T^{\hat{\lambda}}}) \right], \hat{\alpha} \right)$ the solution of the BSDE associated to the terminal condition $\nabla \Tilde{\Phi}\left(\frac{\mathcal{A}_T^{\hat{l}, \hat{\gamma}}}{\mathcal{L}_T^{\hat{\lambda}}}\right)$ and driver $f$. We first need to show that
$\mathcal{E}^{f}_{0,T}\left[\nabla \Tilde{\Phi}\left(\frac{\mathcal{A}_T^{\hat{l}, \hat{\gamma}}}{\mathcal{L}_T^{\hat{\lambda}}}\right)\right]=m.$\\
By the optimality of $\hat{l}$, we get:
\begin{align*}
&\hat{l}m-\textbf{E}\left[\mathcal{L}_T^{\hat{\lambda}}\Tilde{\Phi} \left(\dfrac{\mathcal{A}_T^{\hat{l},\hat{\gamma}}}{\mathcal{L}_T^{\hat{\lambda}}}\right)- \int_0^T\mathcal{A}_s^{\hat{l},\hat{\gamma}} \Tilde{f}(s,\hat{\gamma}_s)ds\right] \\&\qquad \qquad \geq m(\hat{l}+\varepsilon)- \textbf{E}\left[\mathcal{L}_T^{\hat{\lambda}} \Tilde{\Phi}(\dfrac{\mathcal{A}_T^{\hat{l}+\varepsilon,\hat{\gamma}}}{\mathcal{L}_T^{\hat{\lambda}}})- \int_0^T\mathcal{A}_s^{\hat{l}+\varepsilon,\hat{\gamma}} \Tilde{f}(s,\hat{\gamma}_s)ds\right],
\end{align*}
for all $\varepsilon > - \hat{l}.$
Note that  $\mathcal{A}^{l, \gamma}=l \mathcal{A}^{1, \gamma}$ for all $l \in \mathbf{R}$. Since by construction $\Tilde{\Phi}$ is  a.s. convex, we deduce that:

\begin{align*}
m \varepsilon \leq \textbf{E}\left[-\int_0^T \Tilde{f}(s,\hat{\gamma}_s)\mathcal{A}_s^{1,\hat{\gamma}}+\mathcal{A}_T^{1,\hat{\gamma}}\nabla\Tilde{\Phi}\left(\frac{\mathcal{A}_T^{ \hat{l}+\varepsilon,\hat{\gamma}}}{\mathcal{L}_T^{\hat{\lambda}}}\right)\right] \varepsilon.
\end{align*}
We  take in the above inequality $\varepsilon=\frac{1}{n}$ and $\varepsilon=-\frac{1}{n}$. By letting $n$ tend to $\infty$ and using \eqref{HYP} (a) and Lebesgue's Theorem, we finally get:
\begin{align}\label{eq2}
m= \textbf{E}\left[-\int_0^T \Tilde{f}(s,\hat{\gamma}_s)\mathcal{A}_s^{1,\hat{\gamma}}+\mathcal{A}_T^{1,\hat{\gamma}}\nabla\Tilde{\Phi}\left(\frac{\mathcal{A}_T^{ \hat{l},\hat{\gamma}}}{\mathcal{L}_T^{\hat{\lambda}}}\right)\right].
\end{align}


We now introduce the  processes $(\hat{M}, \hat{N}) \in \mathbf{S}_2 \times \textbf{H}_2$, solution of the BSDE associated to the terminal condition $\nabla \Tilde{\Phi}\left(\frac{\mathcal{A}_T^{\hat{l},\hat{\gamma}}}{\mathcal{L}_T^{\hat{\lambda}}}\right)$ and driver
\begin{align}\label{driver}
h(s,\omega,y,z):=-\Tilde{f}(s,\hat{\kappa}_s(\omega),\hat{\vartheta}_s(\omega))+y\hat{\kappa}_s(\omega)+z ^\top \hat{\vartheta}_s(\omega).
\end{align}
Note that $h$ is Lipschitz continuous with respect to $(y,z)$, uniformly in $(s,\omega)$ (see Remark \ref{boundness}). Existence and uniqueness of the solution of the above BSDE is thus guaranteed.

   We apply It\^ o formula to $\mathcal{A}^{1,\hat{\gamma}}\hat{M}$ and obtain:
\begin{align}\label{procdefd}
\mathcal{A}^{1, \hat{\gamma}}_t \hat{M}_t=\mathcal{A}^{1, \hat{\gamma}}_T \nabla \Tilde{\Phi}\left(\frac{\mathcal{A}_T^{ \hat{l},\hat{\gamma}}}{\mathcal{L}_T^{\hat{\lambda}}}\right)- \int_t^T\Tilde{f}(s,\hat{\gamma}_s)\mathcal{A}_s^{1,\hat{\gamma}}ds-\int_t^T\mathcal{A}_s^{1,\hat{\gamma}}\Tilde{N}_sdW_s,
\end{align}
where $\Tilde{N}$ is defined by $\Tilde{N}:=\hat{N}+\hat{M}\hat{\vartheta}.$ Clearly, $\Tilde{N}$ belongs to $\textbf{H}_2$ since $\hat{N} \in \textbf{H}_2$, $\hat{M} \in \mathbf{S}_2$ and $|\!|\hat{\vartheta}|\!|_{\mathbf{R}^d} \leq C$, by Remark \ref{boundness} .


Let us now fix $\gamma=(\kappa, \vartheta) \in \mathcal{V}$. Since $\mathcal{V}$ is convex, we get that for all $\varepsilon \in [0,1],$ $\gamma^{\varepsilon}:=(1-\varepsilon)(\hat{\kappa}, \hat{\vartheta})+ \varepsilon (\kappa, \vartheta) \in \mathcal{V}.$\\
Using now the optimality condition $X_0^{\hat{l}, \gamma^{\varepsilon}, \hat{\lambda}} \geq X_0^{\hat{l}, \hat{\gamma}, \hat{\lambda}} $, the fact that $\hat{l}>0$, the Lagrange's and Lebesgue's Theorems, one can easily show that $\nabla \Tilde{f}(\cdot, \hat{\gamma})$ satisfies:
\begin{align}\label{optim}
0 &\geq \textbf{E}\left[\int_0^T -\mathcal{A}_s^{1,\hat{\gamma}} \left(\hat{K}_s \Tilde{f}(s, \hat{\gamma}_s)+\nabla_p \Tilde{f}(s, \hat{\gamma}_s) \delta \kappa_s+\nabla_q \Tilde{f}(s, \hat{\gamma}_s) ^\top \delta \vartheta_s\right)ds\right.\nonumber\\&\qquad \qquad \left.+\hat{K}_T\mathcal{A}_T^{1,\hat{\gamma}}\nabla \Tilde{\Phi}\left(\frac{\mathcal{A}_T^{ \hat{l},\hat{\gamma}}}{\mathcal{L}_T^{\hat{\lambda}}}\right)\right],
\end{align}
where $(\delta \kappa, \delta \vartheta):=(\kappa-\hat{\kappa}, \vartheta-\hat{\vartheta})$ and $\hat{K}:=\int_0^\cdot \left(\delta \kappa_s-\delta \vartheta_s \hat{\vartheta}_s\right)ds+\int_0^\cdot \delta \vartheta_sdW_s.$\\
By $\eqref{procdefd}$ we have $\mathcal{A}_T^{1,\hat{\gamma}}\nabla \Tilde{\Phi}(\frac{\mathcal{A}_T^{ \hat{l},\hat{\gamma}}}{\mathcal{L}_T^{\hat{\lambda}}})=\mathcal{A}_T^{1,\hat{\gamma}} \hat{M}_T$. Hence inequality  $\eqref{optim}$ can be re-written as follows:
\begin{align}\label{optim1}
0 &\geq \textbf{E}\left[\int_0^T -\mathcal{A}_s^{1,\hat{\gamma}} \left(\hat{K}_s \Tilde{f}(s, \hat{\gamma}_s)+\nabla_p \Tilde{f}(s, \hat{\gamma}_s) \delta \kappa_s\right.\right.\nonumber\\&\left.\left.\qquad \qquad \qquad +\nabla_q \Tilde{f}(s, \hat{\gamma}_s) ^\top \delta \vartheta_s \right)ds+\hat{K}_T\mathcal{A}_T^{1,\hat{\gamma}}\hat{M}_T\right].
\end{align}

The  definition of $\hat{K}$ together with $\eqref{optim1}$ and  It\^ o formula implies:
\begin{align}\label{eq}
0 \leq \textbf{E} \left[\int_0^T\mathcal{A}_s^{1,\hat{\gamma}} \left((\nabla_p \Tilde{f}(s, \hat{\gamma}_s)-\hat{M}_s)\delta \kappa_s +(\nabla_q \Tilde{f}(s, \hat{\gamma}_s)-\hat{N}_s)^\top \delta \vartheta_s \right)ds \right].
\end{align}
We introduce the  map
$\Theta: [0,T] \times \Omega \times \mathbf{R} \times \mathbf{R}^d \mapsto \mathbf{R}$ defined as follows: $$ \Theta:(\omega,t,u,v) \mapsto (\nabla_p \Tilde{f}(\omega,t, \hat{\gamma}_t(\omega))-\hat{M}_t(\omega))(u-\hat{\kappa}_t(\omega)) +(\nabla_q \Tilde{f}(\omega,t, \hat{\gamma}_t(\omega))-\hat{N}_t(\omega)) ^\top (v-\hat{\vartheta}_t(\omega)).$$
By Remark \ref{boundness}, Assumption \ref{HYP} (a) and Theorem 18.19, p.605 in \cite{AB}, there exists a predictable $\Bar{\gamma}$ belonging to $\mathcal{V}$ such that $\Bar{\gamma}=$argmin$ \{\Theta(\cdot, u,v), \,\, (u,v) \in D^2\}$. For each $(t,\omega) \in [0,T] \times \Omega$, define the map $F$ as follows:
\begin{align}
(p,q) \in D_t^2(\omega) \mapsto F(\omega,t,p,q):=\Tilde{f}(\omega,t,p,q)-p \hat{M}_t(\omega)-q^\top \hat{N}_t(\omega).
\end{align}
Note that we have:
$$\Theta(t, \omega, u,v)=\nabla_p F(t,\omega,\hat{\gamma}_t(\omega))(u-\hat{\kappa}_t(\omega))+\nabla_q F(t,\omega,\hat{\gamma}_t(\omega))^\top (v-\hat{\vartheta}_t(\omega)).$$
Since $\eqref{eq}$ holds for all $\gamma \in \mathcal{V}$,  we can take $\Bar{\gamma} \textbf{1}_{\Theta(\cdot, \Bar{\gamma})>0}+\hat{\gamma}\textbf{1}_{\Theta(\cdot, \Bar{\gamma}) \leq 0}$. Hence we derive that, for $dt \otimes dP$- a.e. $(\omega,t) \in \Omega \times [0,T]$, we have:
\begin{align*}
\Theta(t, \omega, u,v) \leq 0, \,\, \forall \,(u,v) \in D_t^2(\omega).
\end{align*}

By a result of convex analysis, this implies that $\hat{\gamma}_t(\omega)$ maximizes $F(\omega,t, \cdot)$ for $dt \otimes dP$- a.e. $(\omega,t) \in \Omega \times [0,T]$ and thus by Assumption \ref{HYP} (d) we get:
\begin{align}\label{rel1}
\Tilde{f}(\cdot, \hat{\gamma})=\hat{\kappa} \hat{M}+\hat{\vartheta}^\top \hat{N}-f(\cdot,\hat{M}, \hat{N}).
\end{align}
The above relation together with the definition of $h$ ( see \eqref{driver}) leads to:
$$h(\cdot,\hat{M}, \hat{N})= f(\cdot,\hat{M}, \hat{N}).$$
Recall that $(\hat{M}, \hat{N})$ represents the solution of the BSDE of terminal condition $\nabla  \Tilde{\Phi}(\frac{\mathcal{A}_T^{ \hat{l},\hat{\gamma}}}{\mathcal{L}_T^{\hat{\lambda}}})$ and driver $h$. Hence by applying the comparison theorem for BSDEs, we get
\begin{align}\label{r1}
(\hat{M}, \hat{N})=(\mathcal{E}_{\cdot,T}^f[\nabla  \Tilde{\Phi}(\frac{\mathcal{A}_T^{ \hat{l},\hat{\gamma}}}{\mathcal{L}_T^{\hat{\lambda}}})], \Hat{\alpha}).
\end{align}
Now, we take the conditional expectation in $\eqref{procdefd}$ and we get:
\begin{align}\label{defM}
\hat{M}_t:=(\mathcal{A}_t^{1, \hat{\gamma}})^{-1} \textbf{E}[-\int_t^T \Tilde{f}(s,\hat{\gamma}_s)\mathcal{A}_s^{1,\hat{\gamma}}+\mathcal{A}_T^{1,\hat{\gamma}}\nabla\Tilde{\Phi}(\frac{\mathcal{A}_T^{ \hat{l},\hat{\gamma}}}{\mathcal{L}_T^{\hat{\gamma}}})| \mathcal{F}_t].
\end{align}
We have cancelled the expectation of $\int_\cdot^T\mathcal{A}_s^{1,\hat{\gamma}}\Tilde{N}_sdW_s$, since by martingale inequalities, it is a martingale.

From $\eqref{eq2}$, $\eqref{r1}$ and $\eqref{defM}$, we derive that  $\mathcal{E}_{0,T}^f \left[\nabla  \Tilde{\Phi}(\frac{\mathcal{A}_T^{ \hat{l},\hat{\gamma}}}{\mathcal{L}_T^{\hat{\lambda}}}) \right]=m$.   Moreover, as $\eqref{rel1}$ holds and $\mathcal{M}_t^{m, \hat{\alpha}}=\mathcal{E}_{t,T}^{f}[\nabla\Tilde{\Phi}(\frac{\mathcal{A}_T^{ \hat{l},\hat{\gamma}}}{\mathcal{L}_T^{\hat{\gamma}}})]$, the first statement of $\eqref{cond1}$ is satisfied.\\

Since $\Tilde{\Phi}$ is a.s.incresing, we derive that $\nabla \Tilde{\Phi}(\frac{\mathcal{A}_T^{\hat{l}, \hat{\gamma}}}{\mathcal{L}_T^{\hat{\lambda}}})\geq 0$ a.s. Also, by construction, $\Tilde{\Phi}$ is a.s. $\textbf{1}$-Lipschitz, which implies that $\nabla \Tilde{\Phi}(\frac{\mathcal{A}_T^{\hat{l}, \hat{\gamma}}}{\mathcal{L}_T^{\hat{\lambda}}})\in [-1,1]$ a.s. We thus conclude that $\nabla \Tilde{\Phi}(\frac{\mathcal{A}_T^{\hat{l}, \hat{\gamma}}}{\mathcal{L}_T^{\hat{\lambda}}}) \in [0,1]$ a.s. and $\mathcal{E}^f_{0,T}[\nabla \Tilde{\Phi}(\frac{\mathcal{A}_T^{\hat{l}, \hat{\gamma}}}{\mathcal{L}_T^{\hat{\lambda}}})]=m$.\\

\noindent{\bf Step 2.}
First, recall that $(Y^{m, \hat{\alpha}},Z^{m, \hat{\alpha}})$ represents the solution of the BSDE with terminal condition $\Phi(\mathcal{M}_T^{m,\hat{\alpha}})$ and driver $g$, where by \textit{Step 1},  $\mathcal{M}_T^{m,\hat{\alpha}}=\nabla \Tilde{\Phi}(\frac{\mathcal{A}_T^{\hat{l}, \hat{\gamma}}}{\mathcal{L}_T^{\hat{\lambda}}})$.

Now, Assumption $\ref{HYP}$ (c) yields
\begin{align}
\Phi(\mathcal{M}_T^{m,\hat{\alpha}})=  \frac{\mathcal{M}_T^{m, \hat{\alpha}} \mathcal{A}_T^{\hat{l},\hat{\gamma}}}{\mathcal{L}_T^{\hat{\lambda}}}-\Tilde{\Phi}(\frac{ \mathcal{A}_T^{\hat{l},\hat{\gamma}}}{\mathcal{L}_T^{\hat{\lambda}}}).
\end{align}

Now, by  using the optimality of $\hat{\lambda}$, i.e. for all $\varepsilon >0$, $X_0^{\hat{l}, \hat{\gamma}, \lambda^\varepsilon} \geq X_0^{\hat{l}, \hat{\gamma}, \hat{\lambda}}$ and similar arguments as in \textit{Step 1}, we get:
\begin{align}\label{egal}
(Y^{m,\hat{\alpha}}, Z^{m,\hat{\alpha}})= (\hat{Y}, \hat{Z}),
\end{align}
where $(\hat{Y}, \hat{Z})$ corresponds to the solution of the BSDE associated to the terminal condition $\Phi(\mathcal{M}_T^{m,\hat{\alpha}})$ and driver $-\Tilde{g}(s,\hat{\mu}_s(\omega),\hat{\nu}_s(\omega))+y\mu_s(\omega)+z ^\top \hat{\nu}_s(\omega).$ Also by the same arguments given at Step 1, $\hat{Y}$ satisfies:
\begin{align}\label{Eq3}
\hat{Y}=(\mathcal{L}^{\hat{\lambda}})^{-1} \textbf{E}_{\cdot}[-\int_\cdot^T \Tilde{g}(s,\hat{\lambda}_s)\mathcal{L}_s^{\hat{\lambda}}+\mathcal{L}_T^{\hat{\lambda}}{\Phi}(\mathcal{M}_T^{m,\hat{\alpha}})].
\end{align}
Since by  $\eqref{egal}$ and $\eqref{Eq3}$ we have $\Hat{Y}_0=Y_0^{m,\Hat{\alpha}}$ and $\mathcal{L}^{\hat{\lambda}}_0=1$,  we obtain:
\begin{align}
Y_0^{m, \Hat{\alpha}}&=\textbf{E} \left[ \mathcal{L}^{\hat{\lambda}}_T \Phi(\mathcal{M}_T^{m,\hat{\alpha}})-\int_0^T \mathcal{L}^{\hat{\lambda}}_s \Tilde{g}(s, \Hat{\lambda})ds \right]=\textbf{E} \left[\mathcal{M}_T^{m,\hat{\alpha}} \mathcal{A}_T^{\Hat{l}, \hat{\gamma}} \right]\nonumber\\&-\textbf{E} \left[ \mathcal{L}^{\hat{\lambda}}_T \Tilde{\Phi}(\frac{\mathcal{A}_T^{\hat{l}, \hat{\gamma}}}{\mathcal{L}_T^{\hat{\lambda}}})+\int_0^T \mathcal{L}^{\hat{\lambda}}_s \Tilde{g}(s, \Hat{\lambda})ds \right].
\end{align}
Now, we appeal to \eqref{eq2}  and since by \textit{Step 1},  $\mathcal{M}_T^{m,\hat{\alpha}}=\nabla \Tilde{\Phi}(\frac{\mathcal{A}_T^{\hat{l}, \hat{\gamma}}}{\mathcal{L}_T^{\hat{\lambda}}})$,  we  get $\textbf{E} \left[\mathcal{M}_T^{m,\hat{\alpha}} \mathcal{A}_T^{\hat{l}, \hat{\gamma}} \right]=\Hat{l}\left(m+\textbf{E} \left[\int_0^T \Hat{\mathcal{A}}_s^{1, \hat{\gamma}} \Tilde{f}(s, \Hat{\gamma})ds \right] \right)$=$m \Hat{l}+\textbf{E} \left[\int_0^T \Hat{\mathcal{A}}_s^{\Hat{l}, \hat{\gamma}} \Tilde{f}(s, \Hat{\gamma})ds \right]$.
From the two above equalities, we finally obtain $$Y_0^{m, \Hat{\alpha}}=\Hat{l}m-\textbf{E} \left[ \mathcal{L}_T^{\hat{\lambda}} \Tilde \Phi(\mathcal{M}^{m,\hat{\alpha}}_T)-\int_0^T \Hat{\mathcal{A}}_s^{\Hat{l}, \hat{\gamma}} \Tilde{f}(s, \hat{\gamma})ds+\int_0^T \mathcal{L}_s^{\hat{\lambda}} \Tilde{g}(s, \Hat{\lambda})ds \right].$$
The above equality together with Proposition \ref{f11} give the desired  result.
\end{proof}

We now show that the existence of an optimal control in the primal problem implies the existence of an optimal control in the dual problem, under the following assumptions:

\begin{assumption}\label{A2}
\item[$(a)$] For each $(t,\omega)$, the maps $\Phi(\omega)$, $f(\omega,t,\cdot)$ and $g(\omega,t,\cdot)$ are $C_b^1$ on $[0,1]$ and $\mathbf{R} \times \mathbf{R}^d$ respectively;

\item[$(b)$] $|\nabla \Phi(\omega,\cdot)| \leq C_{\Phi}(\omega),$ for some $C_\Phi \in \textbf{L}_2(\mathbf{R}).$
\end{assumption}

\begin{proposition}
Let $l>0$ be fixed and assume that there exists $\Hat{m} \in [\mathcal{E}^{f}_{0,T}[0],\mathcal{E}^{f}_{0,T}[1]]$ and $\Hat{\alpha} \in \textbf{A}_{0,\hat{m}}$ such that
\begin{align}
\sup_{m \in [\mathcal{E}^f_{0,T}[0], \mathcal{E}^f_{0,T}[1]]}(ml-\mathcal{Y}_0(m))=\hat{m}l-Y_0^{\hat{m},\hat{\alpha}}.
\end{align}
Then, there exists $(\Hat{\lambda},\Hat{\gamma}) \in \mathcal{U} \times \mathcal{V} $ such that
\begin{align}
\mathcal{Y}_0(\Hat{m})=\Hat{m}l-\mathcal{X}_0(l)=\Hat{m}l-X_0^{l,\Hat{\gamma},\Hat{\lambda}}.
\end{align}
\end{proposition}

\begin{proof}
We use some similar arguments as in \cite{BER}. However, the proof is more involved as we also have to deal with the nonlinear driver $f$. The proof is divided in three steps.

\medskip
\noindent {\bf Step 1.}
Let $\mathcal{M}_\cdot$ be an arbitrary $f$-martingale valued in $[\mathcal{E}_\cdot^f[0], \mathcal{E}_\cdot^f[1]]$ and $\varepsilon \in [0,1]$. We denote by $\mathcal{M}^\varepsilon$ the process defined as $\mathcal{M}_\cdot^\varepsilon:= \mathcal{E}_{\cdot,T}^f \left[\Hat{\mathcal{M}}_T+\varepsilon(\mathcal{M}_T-\mathcal{\Hat{M}}_T) \right]$, where $\Hat{\mathcal{M}}:=\mathcal{M}^{\hat{m}, \hat{\alpha}}$. We set $m_\varepsilon:=\mathcal{M}_0^\varepsilon$ and $(\delta \mathcal{M}, \delta \alpha):=(\mathcal{M}-\hat{\mathcal{M}}, \alpha- \hat{\alpha})$.\\
We now consider the BSDE associated to $\delta \mathcal{M}_T$  and generator: $$h_1(t, \omega,u,v):= \nabla_x f(t, \omega,\hat{\mathcal{M}}_t(\omega), \hat{\alpha}_t(\omega))u+ \nabla_\pi f(t, \omega,\hat{\mathcal{M}}_t(\omega), \hat{\alpha}_t(\omega))^\top v.$$
Since  $\delta \mathcal{M}_T$  belongs to $\textbf{L}_2(\mathcal{F}_T)$  and since by Assumption $\ref{A2}$ on the coefficient $f$, $h$ is uniformly Lipschitz in $(u,v)$ with respect to $(t,\omega)$ , we conclude that the above  BSDE admits an unique solution. This unique solution will be denoted by $(\nabla M, \nabla \alpha)$.

Our aim is to show that $\varepsilon^{-1}( \delta \mathcal{M}^\varepsilon, \delta \alpha^{\varepsilon})$ converges in $\mathbf{S}_2 \times \textbf{H}_2$ as $\varepsilon \rightarrow 0$ to  $(\nabla M, \nabla \alpha).$

First, observe that $\varepsilon^{-1}(\delta \mathcal{M}_s^\varepsilon, \delta \alpha_s^\varepsilon)$ solves the following equation:
\begin{align}
\dfrac{\delta \mathcal{M}_t^\varepsilon}{\varepsilon} = \delta \mathcal{M}_T+ \int_t^T \left(B_s^{\mathcal{M},\varepsilon} \frac{\delta \mathcal{M}_s^\varepsilon}{\varepsilon}+B_s^{\alpha,\varepsilon,\top} \frac{\delta \alpha_s^\varepsilon}{\varepsilon}\right)ds-\int_t^T \frac{\delta \alpha_s^\varepsilon}{\varepsilon}^\top dW_s,
\end{align}
where\\
$$B_s^{\mathcal{M},\varepsilon}:=\int_0^1 \nabla_x f \left(s, \Hat{\mathcal{M}}_s+r \delta \mathcal{M}_s^\varepsilon, \hat{\alpha}_s\right)dr;\,\, B_s^{\alpha,\varepsilon}:=\int_0^1 \nabla_\pi f \left(s, \Hat{\mathcal{M}}_s, \hat{\alpha}_s+r \delta \alpha_s^\varepsilon\right)dr.$$

We now introduce the processes $\Xi^\varepsilon:=\varepsilon^{-1} \delta \mathcal{M}^\varepsilon-\nabla \mathcal{M}$ and  $\Pi^\varepsilon:=\varepsilon^{-1} \delta \alpha^\varepsilon-\nabla \alpha$. We can remark that $(\Xi^\varepsilon,\Pi^\varepsilon)$ solves the BSDE associated to terminal condition $0$ and driver:
$$h_2(t,\omega,u,v):=B_t^{\mathcal{M},\varepsilon}(\omega)u+B_t^{\alpha,\varepsilon}(\omega)^\top v+D_t^\varepsilon(\omega),$$ where
$D_t^\varepsilon:= \nabla \mathcal{M}_t \left(B_t^{\mathcal{M},\varepsilon}-\nabla_x f(t, \Hat{\mathcal{M}_t}, \hat{\alpha}_t)\right)+\nabla \alpha_t ^\top \left(B_t^{\alpha,\varepsilon}-\nabla_\pi f(t, \Hat{\mathcal{M}_t}, \hat{\alpha}_t)\right).$\\
We apply the stability result with  BSDE$(\xi,h_2)$ and BSDE$(\xi,0)$, where $\xi=0$.
We thus get:
\begin{align}\label{convv1}
|\!|\Xi^\varepsilon|\!|_{\mathbf{S}_2}+|\!|\Pi^\varepsilon|\!|_{\textbf{H}_2} \leq C |\!|D^\varepsilon|\!|_{\textbf{H}_2}.
\end{align}
In order to show the convergence of $|\!|D^\varepsilon|\!|_{\textbf{H}_2}$ to $0$ when $\varepsilon \rightarrow 0$, we  prove that $(\mathcal{M}^\varepsilon, \alpha^\varepsilon)$ converges to $(\mathcal{M}, \alpha)$ in $\mathbf{S}_2 \times \textbf{H}_2$. To this purpose, we apply again the stability result for BSDEs and obtain:
\begin{align} \label{convv}
|\!|\mathcal{M}^\varepsilon-\mathcal{M}|\!|^2_{\mathbf{S}_2}+|\!|\alpha^\varepsilon-\alpha|\!|^2_{\textbf{H}_2}\leq C(|\!|\mathcal{M}_T^\varepsilon-\mathcal{M}_T|\!|^2_{\textbf{L}_2})  \rightarrow _{\varepsilon \rightarrow 0} 0.
\end{align}

By \eqref{convv}, Assumption $\ref{A2}$ and the Lebesgue's Theorem, we get that $|\!|D^\varepsilon|\!|_{\textbf{H}_2} \rightarrow 0$ when $\varepsilon \rightarrow  0$. Finally, by $\eqref{convv1}$, we derive that $\varepsilon^{-1}( \delta \mathcal{M}^\varepsilon, \delta \alpha^{\varepsilon})$ converges in $\mathbf{S}_2 \times \textbf{H}_2$ to $(\nabla M, \nabla \alpha)$ as $\varepsilon \rightarrow 0$.

\medskip

\noindent {\bf Step 2.}
We denote by $(Y^\varepsilon, Z^\varepsilon)$ the solution of the BSDE$(g, \Phi(\mathcal{M}_T^\varepsilon))$ and we set $(\hat{Y}, \hat{Z}):=(Y^{m, \hat{\alpha}}, Z^{m, \hat{\alpha}}).$ Using the same arguments as in \textit{Step 2}, one can show that $(\frac{\delta Y^\varepsilon}{\varepsilon},\frac{\delta Z^\varepsilon}{\varepsilon}):= (\frac{Y^\varepsilon-\hat{Y}}{\varepsilon},\frac{Z^\varepsilon-\hat{Z}}{\varepsilon})$ converges in $\mathbf{S}_2 \times \textbf{H}_2$ to the unique solution $(\nabla Y, \nabla Z)$ of the following BSDE:
\begin{align}
\nabla Y_t &= \nabla \Phi(\hat{\mathcal{M}}_T)\delta \mathcal{M}_T+\int_t^T \nabla_yg(s,\hat{Y}_s, \hat{Z}_s) \nabla Y_s ds\nonumber\\&+ \int_t^T \nabla_z g(s,\hat{Y}_s, \hat{Z}_s)^\top \nabla Z_s ds -\int_t^T \nabla Z_s^\top dW_s.
\end{align}

\noindent{\bf Step 3.} Since $(\Hat{m}, \Hat{\alpha})$ is optimal, we have $Y_0^\varepsilon-m_\varepsilon-\Hat{Y}_0+\Hat{m}l \geq 0$, for any $\varepsilon >0$. Dividing now by $\varepsilon >0$ and sending $\varepsilon \rightarrow 0$, we get
\begin{align}\label{ineqqq}
0 &\leq \nabla \Phi(\hat{\mathcal{M}_T}) \delta \mathcal{M}_T+\int_0^T \nabla g(s, \hat{Y}_s, \hat{Z}_s)^\top (\nabla{Y}_s, \nabla{Z}_s)ds-\int_0^T \nabla Z_s^\top dW_s\\&
-l \left( \delta \mathcal{M}_T+ \int_0^T \nabla f(s, \hat{\mathcal{M}}_s, \hat{\alpha}_s)^\top(\nabla{\mathcal{M}}_s, \nabla{\alpha}_s)ds-\int_0^T \nabla \alpha_s ^\top dW_s\right)= \nabla Y_0 - l  \nabla \mathcal{M}_0.\nonumber
\end{align}

We set $\hat{\gamma}_t:= \nabla f(s,\hat{\mathcal{M}}_t, \hat{\alpha}_t)$ and $\hat{\lambda}_t:= \nabla g(s,\hat{Y}_t, \hat{Z}_t)$, which belong to $\mathcal{V}$ and, respectively, $\mathcal{U}$. Since $\hat{\gamma}_t$ ( resp. $\hat{\lambda}_t$) belongs to the superdifferential of $f$ at $(\mathcal{M}_t, \hat{\alpha}_t)$ (resp. the subdifferential of $g$ at $(\hat{Y}_t, \hat{Z}_t)$) we have (see \cite{BP}):
\begin{align}\label{f1}
 f(\cdot,\hat{\mathcal{M}}, \hat{\alpha})= \hat{\kappa}\hat{\mathcal{M}}+ \hat{\vartheta}^\top \hat{\alpha}-\Tilde{f}(\cdot, \hat{\gamma}).
\end{align}
and
\begin{align}\label{f2}
 g(\cdot,\hat{Y}, \hat{Z})= \hat{\mu}\hat{Y}+ \hat{\nu}^\top \hat{Z}-\Tilde{g}(\cdot, \hat{\lambda}).
\end{align}

Now, by applying Ito's formula, we obtain that  $\mathcal{A}^{l,\hat{\gamma}} \nabla {\mathcal{M}}$ and $\mathcal{L}^{\hat{\lambda}} \nabla {Y}$ are martingales. As $\mathcal{L}_0^{\hat{\lambda}}=1$ and  \eqref{ineqqq} holds, we thus obtain:
\begin{align}
\hat{\mathcal{L}}_0 \nabla Y_0-l \nabla \mathcal{M}_0&= E\left[{\mathcal{L}}^{\hat{\lambda}}_T \nabla Y_T-\mathcal{A}^{\hat{l}, \hat{\gamma}}_T \nabla \mathcal{M}_T \right]\nonumber\\&=E \left[\mathcal{L}^{\hat{\lambda}}_T \delta {\mathcal{M}}_T \left(\nabla \Phi(\hat{\mathcal{M}}_T)- \frac{\mathcal{A}^{\hat{l}, \hat{\gamma}}_T}{\mathcal{L}^{\hat{\lambda}}_T} \right)  \right] \geq 0.
\end{align}
Since $\mathcal{M}_T$ can be arbitrary choses with values in $[0,1]$, we obtain that $\hat{\mathcal{M}}_T(\omega)$ minimizes the map $m \in [0,1] \mapsto \Phi(\omega,m)-m \frac{{\mathcal{A}}^{l,\hat{\gamma}}_T}{{\mathcal{L}}^{\hat{\lambda}}_T}(\omega).$ Thus, we obtain: $\hat{\mathcal{M}}_T \mathcal{A}_T^{\hat{l}, \hat{\gamma}}-\mathcal{L}^{\hat{\lambda}}_T \Phi(\hat{\mathcal{M}}_T)=\mathcal{L}^{\hat{\lambda}}_T \Tilde{\Phi}(\frac{\mathcal{A}_T^{\hat{l}, \hat{\gamma}}}{\mathcal{L}^{\hat{\lambda}}_T})$. This inequality together with \eqref{f1}, \eqref{f2} and Ito's formula allow to conclude that $l\hat{m}-\hat{Y}_0=X_0^{l, \hat{\lambda}, \hat{\gamma}}.$ The conclusion follows by Proposition
\ref{f11}.
\end{proof}

\appendix
\section{Appendix}

\textbf{Proof of Proposition} \ref{link}. The proof is standard. We provide it for completeness.
Let $(Y,Z)$ be a supersolution of $BSDE (g, f,\Psi,\mu,\tau)$. Now, the BSDE representation of $\Psi(Y_T)$ implies that it exists $\bar{\alpha} \in \mb{A}_{\tau, \rho}$ such that $\Psi(Y_{T}) = \cal{M}_{T}^{\tau,\rho,\bar{\a}}$, where $\rho:= \cal{E}_{\tau,T}^{f}[\Psi(Y_{T})].$   Since condition $\eqref{Equ2}$ is satisfied, we have $\rho \geq \mu$ a.s.
We define the following stopping time
$$
\sigma^{\bar{\a}}:=\inf \{\tau \leq s \leq T: \cal{M}_{s}^{\tau,\mu,\bar{\a}} = \cal{E}_{s,T}^{f}[0]\} \wedge T,
$$
with the convention $\inf \emptyset=+\infty.$
Recall that  $(Y^{0},Z^{0})$ represents the solution of the BSDE associated to driver $f$ and terminal condition
$0$. We define the control $\Tilde \a$ as follows:
\begin{equation}
\Tilde \a_{s}: = \bar{\a}_{s}\mb{1}_{\{s\leq \sigma^{\bar{\a}}\}} + Z_{s}^{0}\mb{1}_{\{s>\sigma^{\bar{\a}}\}}.
\end{equation}
Note that $\Tilde \a$ belongs to $\textbf{A}_{\tau,\mu}$. The control is constructed in such a way that $\mathcal{M}_\cdot^{\tau,\mu, \Tilde{\alpha}}$ belongs to $[\mathcal{E}_{\cdot,T}^f[0], \mathcal{E}_{\cdot,T}^f[1]]$. We have not considered the hitting time of the process $\mathcal{E}_{\cdot,T}^f[1]$, since clearly $\mathcal{M}_\cdot^{\tau,\mu, \bar{\alpha}} \leq \mathcal{M}_\cdot^{\tau,\rho, \bar{\alpha}}$.
 We can easily remark that $\cal{M}_{T}^{\tau,\rho,\tilde{\a}} \geq \cal{M}_{T}^{\tau,\mu,{\a}}$ a.s. The
monotonocity of $\Phi$ and the identity $\Psi(Y_{T}) = \cal{M}_{T}^{\tau,\rho,\tilde{\a}}$ imply that
\begin{equation}
Y_{T}\geq (\Phi \circ \Psi) (Y_{T})\geq \Phi(\cal{M}_{T}^{\tau,\mu, {\a}}).
\end{equation}
Hence, by the comparison theorem for  BSDEs, we obtain that $Y_{t}\geq \cal{E}_{t,T}^{g}[\Phi(\cal{M}_{T}^{\tau,\mu, {\a}})]$ for $t\in [0,T]$. Conversely, let $\a \in \textbf{A}_{\tau, \mu}$ be such that $Y_{t}\geq \cal{E}_{t,T}^{g}[\Phi(\cal{M}_{T}^{\tau,\mu, {\a}})]$ for $t\in [0,T]$
and suppose that $(Y,Z)$ satifies \eqref{Equ1}. We thus get
$$
\Psi(Y_{T})\geq (\Psi \circ \Phi) (\cal{M}_{T}^{\tau,\mu,{\a}})\geq \cal{M}_{T}^{\tau,\mu,\a}.
$$
Taking the $f$-conditional expectation on both sides, the result follows.

\begin{lemma}\label{sequence}
Fix $\theta, \nu \in \mathcal{T}$, with $\theta \geq \tau, \mu \in \textbf{D}_\tau$ and $\alpha \in \textbf{A}_{\tau,\mu}.$ Then there exists a sequence $(\alpha'_n) \subset \textbf{A}_{\tau,\mu}^{\theta,\alpha}:= \{\alpha' \in \textbf{A}_{\tau,\mu}, \alpha' \textbf{1}_{[0,\theta)}=\alpha\textbf{1}_{[0,\theta)}\}$ such that $\lim_{n\rightarrow \infty} \downarrow \mathcal{E}_{\theta,T}^g[\Phi(\mathcal{M}_T^{\tau,\mu,\alpha'_n})]=\mathcal{Y}_\theta^\alpha(\mathcal{M}_\theta^{\tau,\mu,\alpha})$ a.s.
\end{lemma}

\begin{proof}
In order to obtain the desired result, we only have to prove that $$\{J(\alpha'):=\mathcal{E}_{\theta,T}^g[\Phi(\mathcal{M}_T^{\tau,\mu,\alpha'})],\;\;\alpha' \in \textbf{A}_{\tau,\mu}^{\theta,\alpha}\}$$ is directed downward. Set $A:=\{J(\alpha'_1) \leq J(\alpha'_2)\} \in \mathcal{F}_\theta$ and fix $\alpha'_1, \alpha'_2 \in \textbf{A}_{\tau,\mu}^{\theta,\alpha}$. We denote $\Tilde{\alpha}':=\alpha \textbf{1}_{[0,\theta)}+\textbf{1}_{[\theta,T]}(\alpha'_1 \textbf{1}_A+\alpha'_2 \textbf{1}_{A^c}).$ Note that $\Tilde{\alpha}' \in  \textbf{A}_{\tau,\mu}^{\theta,\alpha}.$ We get: $J(\Tilde{\alpha}')=\mathcal{E}_{\theta,T}[\Phi(\mathcal{M}_\theta^{\tau,\mu,\alpha'_1})\textbf{1}_A+\Phi(\mathcal{M}_\theta^{\tau,\mu,\alpha'_2})\textbf{1}_{A^c}]=\min \{J(\alpha'_1), J(\alpha'_2)\}.$
\end{proof}

\begin{theorem}\label{monotonicity}
Fix $t \in [0,T]$. The map $\cal{Y}_{t}:\mu\to \cal{Y}_{t}(\mu); \,\ \textbf{D}_t \mapsto \textbf{L}_2$; is non-decreasing, i.e. for all $\mu_{1},\mu_{2}\in \textbf{D}_t$, we have $\cal{Y}_{t}(\mu_{1})\leq \cal{Y}_{t}(\mu_{2})$ on $\{\mu_1 \leq \mu_2 \}$ and $\cal{Y}_{t}(\mu_{1})\geq \cal{Y}_{t}(\mu_{2})$ on $\{\mu_1 \geq \mu_2 \}$.
\end{theorem}

\begin{proof}
The proof is divided in two steps.\\

\noindent \textbf{Step 1.} 
We set $\Tilde{\mu}_1:= \mu_{1} \wedge \mu_{2}$ and $\Tilde{\mu}_2:=  \mu_{1}  \vee  \mu_{2}.$ Remark that $\Tilde{\mu}_1$ and $\Tilde{\mu}_2$ belong to $\textbf{D}_t$.

By Lemma \ref{sequence}, we know that it exists $ \a^{n} \in \mb{A}_{t, \Tilde{\mu}_{2}}$  s.t. $\cal{E}_{t,T}^{g}[\Phi(\cal{M}_{T}^{\Tilde{\mu}_{2},\a^{n}})]\to \cal{Y}_{t}(\Tilde{\mu}_{2})$  a.s.

Fix $n \in \mathbb{N}$. We define $\tilde{\a}^{n} \in \cal{A}_{t, \Tilde{\mu}_{1}}$  as follows:
$$
\tilde{\a}^{n}_{s}:= \a_{s}^{n}\mb{1}_{s\leq \tau} + Z_{s}^{0}\mb{1}_{s>\tau},
$$
where $\tau:= \inf \{ t \leq s \leq T: \cal{M}_{s}^{\Tilde{\mu}_{1},\a^{n}} = \cal{E}_{s,T}^{f}[0]\} \wedge T$, with the convention $\inf \emptyset =+\infty.$ Recall that $Z^0$ is the associated control to the the BSDE with terminal condition $0$ and driver $f$.

By construction of $\tilde{\a}^{n}$, we have $\cal{M}_{T}^{\Tilde{\mu}_{1},\a^{n}} \in [0,1]$ a.s. Now, by using the fact that $\Phi$ in nondecreasing
and the comparison theorem for BSDEs, we obtain:
$$
\cal{E}_{t,T}^{g}[\Phi(\cal{M}_{T}^{\tilde{\mu}_{1},\tilde{\a}^{n}})]\leq \cal{E}_{t,T}^{g}[\Phi(\cal{M}_{T}^{\tilde{\mu}_{2},\a^{n}})]\; \text{a.s.}
$$
which implies
\begin{equation}
\cal{Y}_{t}(\Tilde{\mu}_{1})\leq \cal{E}_{t,T}^{g}[\Phi(\cal{M}_{T}^{\Tilde{\mu}_{2},\a^{n}})]\; \text{a.s.}
\end{equation}
By letting $n\to\infty$ in the above relation, we obtain $\cal{Y}_{t}(\Tilde{\mu}_{1})\leq \cal{Y}_{t}(\Tilde{\mu}_{2})$  a.s.\\

\noindent{\bf Step 2.}
We define $A:=\{\mu_1 \leq \mu_2 \} \in \mathcal{F}_t.$  Let us show that $\mathcal{Y}_t(\Tilde{\mu}_1)=\mathcal{Y}_t(\mu_1)\textbf{1}_A+\mathcal{Y}_t(\mu_2)\textbf{1}_{A^c}.$  For all $\alpha_{i} \in \textbf{A}_{t,\mu_{i}}$, $i=1,2$, we set $\Tilde{\alpha}:= \textbf{1}_{[t,T]}\left(\alpha_1 \textbf{1}_A+\alpha_2 \textbf{1}_{A^c} \right) \in \textbf{A}_{t,{\Tilde{\mu}}_1}.$ Bt the zero-one law for $f$- conditional expectations, we get $\mathcal{E}_{t,T}^g[\Phi(\mathcal{M}_T^{\Tilde{\mu}_1, \Tilde{\alpha}})]=\mathcal{E}_{t,T}^g[\Phi(\mathcal{M}_T^{\mu_1, \alpha_1})]\textbf{1}_A+\mathcal{E}_{t,T}^g[\Phi(\mathcal{M}_T^{\mu_2, \alpha_2})]\textbf{1}_{A^c}$ and by arbitrariness of $\alpha_i$, $i=1,2$, we derive that $\mathcal{Y}_t(\Tilde{\mu}_1) \leq \mathcal{Y}_t(\mu_1)\textbf{1}_A+ \mathcal{Y}_t(\mu_2)\textbf{1}_{A^c}.$ In order to show that $\mathcal{Y}_t(\Tilde{\mu}_1) \geq \mathcal{Y}_t(\mu_1)\textbf{1}_A+ \mathcal{Y}_t(\mu_2)\textbf{1}_{A^c},$ we use the previous equality with $\alpha_1 := \Tilde{\alpha} \textbf{1}_A +\Tilde{\alpha}_1\textbf{1}_{A^c}$ and $\alpha_2:= \Tilde{\alpha}_2 \textbf{1}_{A}+ \Tilde{\alpha} \textbf{1}_{A^c}$ , for all  $\Tilde{\alpha}  \in \textbf{A}_{t, \Tilde{\mu}_1}$, $\Tilde{\alpha}_1 \in \textbf{A}_{t, \mu_1}$ and $\Tilde{\alpha}_2 \in \textbf{A}_{t, \mu_2}$. Similarly, one can prove that
$\mathcal{Y}_t(\Tilde{\mu}_2)=\mathcal{Y}_t(\mu_2)\textbf{1}_A+\mathcal{Y}_t(\mu_1)\textbf{1}_{A^c}.$

From Step 1 and Step 2, the result follows.

\end{proof}

Using the same arguments as in Step 2 of the above proof, one can easily show:

\begin{lemma}\label{R}
Fix $t \in [0,T]$. We have $\mathcal{Y}_t(\mu_1\textbf{1}_A+\mu_2\textbf{1}_{A^c})=\mathcal{Y}_t(\mu_1)\textbf{1}_A+\mathcal{Y}_t(\mu_2)\textbf{1}_{A^c}$, for all $A \in \mathcal{F}_t$, $\mu_1, \mu_2 \in \textbf{D}_t$.
\end{lemma}

We now recall the following result, which can be found in  \cite{BER}.

\begin{proposition}\label{estimation}
Let the Assumption \ref{f} (with $g$ instead $f$) holds. Then:
\begin{itemize}
\item[$(i)$] There exist $\chi_g \in \mathbf{L}_2$ and $C>0$ which only depends on $C_g$ and $T$ such that:
$$\esssup_{\xi \in \textbf{L}_0([0,1])}|\mathcal{E}_{t,T}^g[\xi] \leq C(1+E_t[|\chi_g|^2])|^{\frac{1}{2}}), \,\, \, 0 \leq t \leq T. $$
\item[$(ii)$] For some $\xi \in \textbf{L}_2$ and $t \in [0,T]$, consider a family $(\xi^\varepsilon)_{\varepsilon \geq 0} \subset \textbf{L}_0(\mathbf{R}^d)$ satisfying $|\xi^\varepsilon| \leq \xi$ and $\xi^\varepsilon \in \textbf{L}_0(\mathcal{F}_{(t+\varepsilon) \wedge T}),$ for any $\varepsilon>0.$ Then, there exists a family $(\eta_\varepsilon)_{\varepsilon >0} \subset \textbf{L}_0(\mathbf{R})$ which converges to $0$ $\mathbb{P}$ - a.s. as $\varepsilon \rightarrow 0$ such that:
\begin{align*}
\left|\mathcal{E}^g_{t,t+\varepsilon}[\xi^\varepsilon]-E_t[\xi^\varepsilon] \right| \leq \eta_{\varepsilon}, \,\,\, \forall \varepsilon \in [0,T-t].
\end{align*}
\item[$(iii)$] Let $(\xi^\varepsilon)_{\varepsilon>0}$ and $t \in [0,T]$ be as in $(ii)$. Then, there exists a family $(\eta_\varepsilon)_{\varepsilon>0} \subset \textbf{L}_0(\mathbf{R})$ which converges to $0$ a.s. as $\varepsilon \rightarrow 0$ such that
\begin{align*}
\left|\mathcal{E}^g_{t-\varepsilon,t}[\xi^\varepsilon]-E_t[\xi^\varepsilon] \right| \leq \eta_{\varepsilon}, \,\,\, \forall \varepsilon \in [0,t].
\end{align*}
 \end{itemize}

\end{proposition}


\begin{thebibliography}{15}

%
%
%
%
%
%
%
%
%
%
%
%
%
%
%
%
%
%
%
%
%
%
%

\bibitem{AB} C.D. Aliprantis and K.C. Border, \textsc{Infinite dimensional analysis: a hitchhiker's guide}, {\em Springer Verlag}, 2006.

\bibitem{BP} V. Barbu, T. Precupanu, \textsc{Convexity and Optimization in Banach Spaces}, {\em Springer Monographs in Mathematics}, 4th Edition, 2012.

\bibitem{BEK} Barrieu, Pauline and El Karoui , Nicole, \textsc{Optimal derivatives design under dynamic risk measures} In: Yin, George and Zhang, Qing , (eds.) Mathematics of Finance. Contemporary mathematics (351). American Mathematical Society , Providence, USA, 13-26. ISBN 9780821834121, 2004.

\bibitem{B} B.Bouchard, \textsc{Stochastic targets with mixed diffusion processes}, {\em Stochastic processes and their Applications}, 101:273–302, 2002.
    
\bibitem{B1} B. Bouchard, \textsc{Portfolio management under risk constraints}, {\em Lecture notes}.

\bibitem{BBC}  B. Bouchard, G. Bouveret, J.F. Chassagnuex, \textsc{A backward dual representation for the quantile hedging of Bermudean options}, preprint.
    
\bibitem{BER} B. Bouchard, R. Elie,  A. R\'eveillac, \textsc{BSDEs with weak terminal condition}, {\em Annals of Probability}, 2015.
\bibitem{BET} B. Bouchard, R. Elie, N. Touzi, \textsc{Stochastic target problems with controlled loss}, {\em SIAM Journal on Control and Optimization}, 48(5):3123-3150, 2009.
    
\bibitem{BPT} B. Bouchard, D. Possamai, X.Tan, \textsc{A general Doob-Meyer-Mertens decomposition for g-supermartingale systems}, preprint.

\bibitem{BDHPS} Ph. Briand, B. Delyon, Y. Hu, E. Pardoux, L. Stoica,
\textsc{ $L^p$ solutions of backward stochastic differential equations}, {\em Stochastic Processes and Applications}, 108, pg.109–129, 2003.

\bibitem{B} Ph. Briand, BSDEs and viscosity solutions of semilinear PDEs, {\em Stochastics Stochastics} Rep. 64 (1998), 1–32.

\bibitem{DM1} C. Dellacherie and P.-A. Meyer (1975).
 \textit{Probabilit\'es et Potentiel, Chap. I-IV}. Nouvelle \'edition. Hermann. {\bf MR}{0488194}

\bibitem{DM2}
C. Dellacherie and P.-A. Meyer (1980):
 \textit{Probabilit\'es et Potentiel, Th\'eorie des Martingales, Chap. V-VIII}. Nouvelle \'edition. Hermann.
{\bf MR}{0566768}
\bibitem{DEP} Dumitrescu R., Elie R., Possamai D., \textsc{Weak reflected BSDEs and approximative hedging for American options}, \textit{in preparation}.


\bibitem{FL1} F\"{o}llmer H. and Leukert. P., \textsc{Quantile hedging}, \textit{Finance and Stochastics}, 3(3):252-273, 1999.
￼

\bibitem{FL2}  F\"{o}llmer   H. and Leukert. P., \textsc{Efficient hedging: cost versus shortfall risk}, \textit{Finance and Stochastics}, 4(2):117-146, 2000. 

\bibitem{M} L. Moreau, \textsc{Stochastic target problems with controlled expected loss in jump diffusion models}, {\em SIAM Journal on Control and Optimization}, 49:2577-2609, 2011.
\bibitem{P} S. Peng,
\textsc{Monotonic limit theorem of BSDE and nonlinear decomposition theorem of doob-meyer's type},
113:473-499, {\em Probability thoery and related fields.},1999.

\bibitem{15}  Peng S.,
\textsc{Nonlinear expectations, nonlinear evaluations and risk measures},
165-253, {\em Lecture Notes in Math.}, 1856, Springer, Berlin, 2004.
\bibitem{QS} M.-C. Quenez  and A.Sulem, \textsc{BSDEs with jumps, optimization and applications to dynamic risk measures}, \textit{Stochastic Processes and their Applications} 123 (2013), pp. 3328-3357.
\bibitem{R} R. T. Rockafellar, \textsc{Convex analysis}. Number 28 in Princeton Mathematical Series. Princeton: Princeton University Press, 1970.


%
%
%
%
%
%
%
%

%
%
%
%
%
%
%
%
%
%
%
%
%
%
%
%
%
%
\end{thebibliography}
\end{document}